\pdfoutput=1
\documentclass[11pt,a4paper]{amsart}
\usepackage[margin=1in]{geometry}
\usepackage[utf8]{inputenc}
\usepackage[T1]{fontenc}
\usepackage[dvipsnames]{xcolor}
\usepackage{microtype}
\usepackage{fnpct}

\usepackage{amsmath}
\usepackage{amssymb}
\usepackage{eucal}
\usepackage{mathrsfs}
\usepackage{tikz-cd}
\renewcommand{\injlim}{\varinjlim}
\renewcommand{\projlim}{\varprojlim}

\usepackage{enumitem}
\usepackage{booktabs}
\usepackage[pdfusetitle,colorlinks]{hyperref}
\hypersetup{bookmarksdepth=2,pdfencoding=unicode,allcolors=MidnightBlue}

\usepackage{zref-clever}
\zcsetup{abbrev=false,cap=true,nameinlink=false,sort=false,lang=english}
\newcommand{\cref}[1]{\zcref{#1}}
\newcommand{\Cref}[1]{\zcref[S]{#1}}
\zcsetup{pairsep={ and~},lastsep={, and~}}
\zcRefTypeSetup{equation}{Name-sg=,Name-pl=,refbounds={(,,,)}}
\AddToHook{env/equation/begin}{\zcsetup{countertype={equation=equation}}}
\AddToHook{env/align/begin}{\zcsetup{countertype={equation=equation}}}
\zcRefTypeSetup{item}{Name-sg=,Name-pl=,refbounds={(,,,)}}
\newlist{conenum}{enumerate}{1}
\setlist[conenum,1]{label=(\roman*),ref=\roman*}
\zcRefTypeSetup{conenumi}{Name-sg=,Name-pl=,refbounds={(,,,)}}

\NewDocumentCommand{\newzctheorem}{momo}{\IfValueTF{#4}
  {\newtheorem{#1}{#3}[#4]}
  {\IfValueTF{#2}
    {\AddToHook{env/#1/begin}{\zcsetup{countertype={#2=#1}}}\newtheorem{#1}[#2]{#3}}
    {\newtheorem{#1}{#3}}}}
\numberwithin{equation}{section}
\theoremstyle{plain}
\newzctheorem{Theorem}{Theorem}

\zcRefTypeSetup{Theorem}{Name-sg=Theorem,Name-pl=Theorems}
\newzctheorem{theorem}[equation]{Theorem}
\newzctheorem{conjecture}[equation]{Conjecture}
\zcRefTypeSetup{conjecture}{Name-sg=Conjecture,Name-pl=Conjectures}
\newzctheorem{corollary}[equation]{Corollary}
\newzctheorem{lemma}[equation]{Lemma}
\newzctheorem{proposition}[equation]{Proposition}

\theoremstyle{definition}
\newzctheorem{definition}[equation]{Definition}
\newzctheorem{example}[equation]{Example}
\newzctheorem{question}[equation]{Question}
\zcRefTypeSetup{question}{Name-sg=Question,Name-pl=Questions}
\newzctheorem{assumption}[equation]{Assumption}
\zcRefTypeSetup{assumption}{Name-sg=Assumption,Name-pl=Assumptions}

\theoremstyle{remark}
\newzctheorem{aside}[equation]{Aside}
\zcRefTypeSetup{aside}{Name-sg=Aside,Name-pl=Asides}
\newzctheorem{remark}[equation]{Remark}
\newzctheorem{summary}[equation]{Summary}
\zcRefTypeSetup{summary}{Name-sg=Summary,Name-pl=Summaries}

\let\oldAA\AA\let\AA\relax
\let\oldL\L\let\L\relax
\let\oldSS\SS\let\SS\relax
\let\oldtop\top\let\top\relax

\newcommand{\NN}{\mathbf{N}}
\newcommand{\ZZ}{\mathbf{Z}}
\newcommand{\QQ}{\mathbf{Q}}
\newcommand{\RR}{\mathbf{R}}
\newcommand{\CC}{\mathbf{C}}
\newcommand{\HH}{\mathbf{H}}
\newcommand{\SS}{\mathbf{S}}

\newcommand{\AA}{\mathbb{A}}

\newcommand{\DD}{\mathbb{D}}

\newcommand{\PP}{\mathbb{P}}

\newcommand{\E}{\mathrm{E}}

\newcommand{\cn}{\textnormal{cn}}

\newcommand{\fin}{\textnormal{fin}}

\newcommand{\top}{\textnormal{top}}

\newcommand{\ana}{\textnormal{ana}}
\newcommand{\st}{\textnormal{st}}

\newcommand{\D}{\operatorname{D}}
\newcommand{\BD}{\operatorname{BD}}

\newcommand{\CMon}{\operatorname{CMon}}

\newcommand{\Fun}{\operatorname{Fun}}
\newcommand{\GL}{\operatorname{GL}}
\newcommand{\HC}{\operatorname{HC}}

\newcommand{\KU}{\operatorname{KU}}
\newcommand{\Map}{\operatorname{Map}}
\newcommand{\Mat}{\operatorname{Mat}}

\newcommand{\ev}{\operatorname{ev}}

\newcommand{\NK}{\operatorname{NK}}

\newcommand{\Perf}{\operatorname{Perf}}
\newcommand{\Post}{\operatorname{Post}}
\newcommand{\Pro}{\operatorname{Pro}}
\newcommand{\Pt}{\operatorname{Pt}}
\newcommand{\Shv}{\operatorname{Shv}}
\newcommand{\Sh}{\operatorname{Sh}}
\newcommand{\Sp}{\operatorname{Sp}}
\newcommand{\Spec}{\operatorname{Spec}}
\newcommand{\Proj}{\operatorname{Proj}}
\newcommand{\Vect}{\operatorname{Vect}}
\newcommand{\cofib}{\operatorname{cofib}}
\newcommand{\coker}{\operatorname{coker}}

\newcommand{\fil}{\operatorname{fil}}
\newcommand{\gp}{\operatorname{gp}}

\newcommand{\hyp}{\operatorname{hyp}}
\newcommand{\id}{\operatorname{id}}
\newcommand{\ko}{\operatorname{ko}}
\newcommand{\ksp}{\operatorname{ksp}}
\newcommand{\ku}{\operatorname{ku}}

\newcommand{\op}{\operatorname{op}}
\newcommand{\cf}{\operatorname{cf}}

\newcommand{\X}{\mathord{-}}

\newcommand{\sol}{\operatorname{\vcenter{\hbox{\scalebox{0.5}{\(\blacksquare\)}}}}}
\newcommand{\con}{\operatorname{con}}
\newcommand{\dis}{\operatorname{dis}}
\newcommand{\Prodis}{\operatorname{Pro-dis}}

\newcommand{\cat}[1]{\mathcal{#1}}
\newcommand{\cst}[1]{\underline{#1}}
\newcommand{\Cls}[1]{\mathscr{#1}}
\newcommand{\idl}[1]{\mathfrak{#1}}
\newcommand{\Cat}[1]{\mathsf{#1}}

\title{(Semi)topological \texorpdfstring{\(K\)}{K}-theory via solidification}
\author{Ko Aoki}
\address{Department of Mathematical Sciences,
  University of Copenhagen, Denmark
}
\email{aoki@math.ku.dk}
\date{\today}

\begin{document}

\begin{abstract}
  Clausen--Scholze introduced
  the notion of solid spectrum
  in their condensed mathematics program.
  We demonstrate that
  the solidification of algebraic \(K\)-theory
  recovers two known constructions:
  the semitopological \(K\)-theory
  of a real (associative) algebra
  and 
  the topological (aka operator) \(K\)-theory
  of a real Banach algebra.
\end{abstract}

\maketitle
\setcounter{tocdepth}{1}
\tableofcontents

\section{Introduction}\label{s:intro}

Consider a real Banach algebra~\(A\).
Its \(K\)-theory
typically refers to
its topological (aka operator) \(K\)-theory~\(K^{\top}(A)\).
Nevertheless, we can still take
its (nonconnective) algebraic \(K\)-theory~\(K(A)\).
There is a comparison map
\(K(A)\to K^{\top}(A)\),
which induces an equivalence on \(\pi_0\).
These two are in general very different:
When \(A=\CC\),
the homotopy group of \(K^{\top}(\CC)=\KU\)
is just~\(\ZZ\) or~\(0\)
depending on the parity of the degree,
whereas its higher (algebraic) \(K\)-groups are enormous
and lower ones are zero.

In this paper,
we explain how to get~\(K^{\top}\) from~\(K\)
using the language of condensed mathematics.
A condensed spectrum is
a functor
from extremally disconnected sets
(i.e., a retract of \(\beta I\) for some set~\(I\))
to spectra
satisfying certain conditions.
For any reasonable topological ring~\(A\),
we get the corresponding condensed ring,
which we still write as~\(A\),
by considering \(S\mapsto\Cls{C}(S;A)\).
With this,
we can consider \(K(A)\) as a condensed spectrum
by considering \(S\mapsto K(\Cls{C}(S;A))\).
One compelling tool in condensed mathematics
is the operation of solidification,
a certain completion process
applied to condensed spectra.
We prove the following:

\begin{Theorem}\label{com_ban}
  Let \(A\) be a complex Banach algebra.
  The solidification of
  its connective \(K\)-theory
  \(K^{\cn}(A)^{\sol}\)
  is discrete
  (i.e., its condensed structure is trivial)
  and
  canonically equivalent to
  the connective part
  of \(K^{\top}(A)\).
  The condensed spectrum
  \(K(A)^{\sol}[\beta^{-1}]\) is discrete
  and
  canonically equivalent to~\(K^{\top}(A)\),
  where \(\beta\in K_{2}^{\top}(\CC)\simeq\ZZ\)
  is a generator.
\end{Theorem}

Note that
we can also regard a complex algebra~\(A\) (without topology)
as a condensed \(\CC\)-algebra
by considering
\(S\mapsto A\otimes_{\CC}\Cls{C}(S;\CC)\).
In that case,
we recover the semitopological \(K\)-theory
of Friedlander--Walker~\cite{FriedlanderWalker01C}\footnote{We here refer to what they introduced in~\cite{FriedlanderWalker01C},
  not to what they called ``semitopological \(K\)-theory'' in that paper;
  see~\cite{FriedlanderWalker05} for the history.
}
and Blanc~\cite{Blanc16}:

\begin{Theorem}\label{com_cat}
  Let \(A\) be a complex associative algebra.
  We equip \(A\) with the condensed structure
  coming from the usual topology on~\(\CC\).
  Then \(K(A)^{\sol}\) is discrete
  and canonically coincides with
  the semitopological \(K\)-theory~\(K^{\st}(A)\).
\end{Theorem}

\begin{remark}\label{1c99e9dff2}
  In \cref{com_ban,com_cat},
  the use of solidification is
  nonessential;
  as it is discrete,
  we can just say that it is the discretization,
  i.e., the universal discrete object receiving a map from it.
  However, unlike solidification,
  discretization a~priori does not exist;
  see \cref{xh0gyx}.
\end{remark}

The usual constructions of these invariants
use certain information about
\(K(\Cls{C}([0,1];A))\) for a complex Banach algebra~\(A\)
or \(K(A\otimes_{\CC}\Cls{C}([0,1];\CC))\) for a complex algebra~\(A\).
What is notable about these results is that
they show we only need information
about \(K(\Cls{C}(S;A))\) or \(K(A\otimes_{\CC}\Cls{C}(S;\CC))\)
for extremally disconnected sets~\(S\).
While \([0,1]\) is geometrically nicer than extremally disconnected sets,
\(\Cls{C}(S;A)\), where \(S\) is extremally disconnected,
is ring-theoretically typically easier to study;
e.g., see~\cite[Remark~5.1 and Theorem~5.14]{k-ros-1} for the case when \(A=\CC\).

As an instance of this approach,
we explain some reformulations of
the following conjecture of Rosenberg
made in~\cite[page~458]{Rosenberg90}
and~\cite[Conjecture~2.2]{Rosenberg97}:

\begin{conjecture}[Rosenberg]\label{ros}
  Let \(A\) be a (real) C*-algebra.
  \begin{enumerate}
    \item\label{i:ros_pos}
      For \(n>0\),
      the map
      \(K_{*}(A;\ZZ/n)\to K^{\top}_{*}(A;\ZZ/n)\)
      is an isomorphism for \({*}\geq0\).
    \item\label{i:ros_neg}
      The tautological map
      \(K_{*}(A)\to K_{*}(\Cls{C}([0,1];A))\)
      is an isomorphism for \({*}\leq0\).
  \end{enumerate}
\end{conjecture}

We use our main results
to formulate a similar problem,
which implies \cref{ros},
without mentioning \(K^{\top}\) or \([0,1]\);
see \cref{yos}.

\begin{remark}\label{xdpyca}
  We can also study the \emph{semitopological \(K\)-theory}
  of a complex Banach algebra~\(A\):
  We define \(K^{\cn,\st}(A)\) as \(K^{\cn}(A)^{\sol}\),
  which is discrete by \cref{com_ban}.
  It also coincides with the geometric realization
  of \(K^{\cn}(\Cls{C}(\Delta^{\bullet};A))\)
  by \cref{hi_ban}.
  If \cref{i:yos_neg} of \cref{yos} is true for~\(A\),
  the solidification~\(K(A)^{\sol}\)
  is also discrete and
  we can define \(K^{\st}(A)\) as \(K(A)^{\sol}\)
  (or we can unconditionally consider \(K(A)^{\sol}(*)\)).

  Furthermore,
  we can define its \emph{integral cohomology} as
  \begin{equation*}
    \ZZ^{\top}(A)
    =
    \cofib(\Sigma^2K^{\st}(A)\xrightarrow{\beta}K^{\st}(A))
    \in\D(\ZZ).
  \end{equation*}
  This captures the information
  lost in the process of obtaining \(K^{\top}(A)\)
  by inverting~\(\beta\).
Note that this ``cohomology theory''
  for C*-algebras is not \(\Cls{K}\)-stable\footnote{We write \(\Cls{K}\) for the C*-algebra of compact operators
    on a separable Hilbert space.
  }\footnote{Note that
    this invariant makes sense for nonunital C*-algebras
    by \cref{exc}.
  } since \(\ZZ^{\top}(\Cls{K})\) vanishes
  by Suslin--Wodzicki's theorem~\cite{SuslinWodzicki92}.
Still, it is \(\Mat_{2}\)-stable by definition.

Consider the case where \(A=\Cls{C}(X;\CC)\) for a compactum~\(X\).
  By applying Cortiñas--Thom's theorem~\cite{CortinasThom12},
  we find that \(K^{\st}(A)\simeq\Gamma(X;\cst{\ku})\)
  and thus \(\ZZ^{\top}(A)\simeq\Gamma(X;\cst{\ZZ})\).
\end{remark}

\subsection*{Organization}\label{ss:outline}

We give background on condensed mathematics in \cref{s:cm}
in order to study
the solidified \(K\)-theory
of a condensed \(\E_{1}\)-algebra over~\(\RR\).
It has nice properties under reasonable assumptions:
One is truncatedness
and the other is discreteness,
which we see in \cref{s:tr,s:dis},
respectively.
We then move on to the comparison results;
we prove (the real versions of) \cref{com_cat,com_ban}
in \cref{s:stop,s:op}, respectively.
Lastly, in \cref{s:ros},
we pose a conjecture that implies \cref{ros}
and demonstrate the easiest case thereof.

\subsection*{Acknowledgments}\label{ss:ack}

I thank Peter Scholze for helpful discussions
and useful comments on a draft.
Except for \cref{ss:counterexample},
this work was carried out at the Max Planck Institute for Mathematics,
which I thank for its financial support and hospitality.
The work in \cref{ss:counterexample}
was carried out at the University of Copenhagen,
where I was supported by the Danish National Research Foundation
through the Copenhagen Center for Geometry and Topology (DNRF151).

Most of the main results here were first announced
in my talk at the Arbeitstagung held at the Institute in June 2023.
I apologize for the delay in writing this paper.

\section{Background on condensed mathematics}\label{s:cm}

In this section,
we explain the basic condensed mathematics
we need in this paper.
We review the notions of
condensed object, discrete object,
and solid spectrum
in \cref{ss:con,ss:dis,ss:sol}, respectively.
\Cref{ss:kr}
contains a useful observation concerning (topos-theoretic) points.

\subsection{Recollection: condensed animas}\label{ss:con}

We recall the notion of condensed objects.
We follow~\cite[Section~2]{Condensed}
and~\cite[Section~2.1]{Mann}
for subtleties about cutoff cardinals.

For an infinite cardinal~\(\kappa\),
we write \(\Cat{Cpt}_{\kappa}\) for the category
of compacta of weight~\(<\kappa\)\footnote{When \(\kappa\) is regular,
  this coincides with
  the category of \(\kappa\)-cocompact compacta;
  see, e.g.,~\cite[Proposition~2.11]{k-ros-1}.
}.
We equip \(\Cat{Cpt}_{\kappa}\) with the topology
such that
a sieve on~\(X\) is a cover
if and only if it contains a \emph{finite} set \(\{X_i\to X\}_{i\in I}\)
such that \(\coprod_{i\in I}X_i\to X\) is surjective.
We write \(\Cat{ConAni}_{\kappa}\)
for the \(\infty\)-category
of anima-valued hypersheaves on \(\Cat{Cpt}_{\kappa}\)
and call its objects \emph{\(\kappa\)-condensed anima}.
For a presentable \(\infty\)-category~\(\cat{C}\),
we define the \(\infty\)-category of \emph{\(\kappa\)-condensed objects}
in~\(\cat{C}\) to be
\begin{equation*}
  \Cat{ConAni}_{\kappa}\otimes\cat{C}
  \simeq
  \Shv^{\hyp}(\Cat{Cpt}_{\kappa};\cat{C}).
\end{equation*}
We write \(\Cat{TDis}_{\kappa}\subset\Cat{Cpt}_{\kappa}\)
for the full subcategory
spanned by totally disconnected ones.
Since \(\Cat{TDis}_{\kappa}\) is a basis of \(\Cat{Cpt}_{\kappa}\),
by, e.g.,~\cite[Corollary~A.7]{Aoki23},
an object of \(\Cat{ConAni}_{\kappa}\)
is equivalently described as a hypersheaf on
\(\Cat{TDis}_{\kappa}\).

Suppose that \(\kappa\) is strong limit;
i.e.,
a nonzero cardinal
satisfying \(2^{\lambda}<\kappa\)
for all \(\lambda<\kappa\).
In this case,
a compactum is inside \(\Cat{Cpt}_{\kappa}\)
if and only if its underlying set 
is of cardinality~\(<\kappa\)
and hence
the full subcategory \(\Cat{EDis}_{\kappa}\subset\Cat{Cpt}_{\kappa}\)
spanned by extremally disconnected ones
also forms a basis.
What is notable is that
the (hyper)sheaf condition on \(\Cat{EDis}_{\kappa}\) is
easily described due to Gleason's theorem~\cite{Gleason58}:
A covariant functor from \(\Cat{EDis}_{\kappa}\)
to an \(\infty\)-category having limits
is a sheaf if and only if
it carries finite coproducts to products.
Therefore, for a presentable \(\infty\)-category~\(\cat{C}\),
a \(\kappa\)-condensed object
is equivalently described as a functor
\(\Cat{EDis}_{\kappa}^{\op}\to\cat{C}\)
preserving finite products.

For infinite cardinals \(\kappa\leq\lambda\),
the restriction functor \(\Cat{ConAni}_{\lambda}\to\Cat{ConAni}_{\kappa}\)
has a left adjoint.
The \(\infty\)-category of \emph{condensed objects} in~\(\cat{C}\)
is
\begin{equation}
  \label{e:lwgc9}
  \injlim_{\kappa}
  \Cat{ConAni}_{\kappa}\otimes\cat{C},
\end{equation}
where \(\kappa\) runs over all infinite cardinals.
We are ultimately interested in this \(\infty\)-category
and
it is not the case that the notion of \(\kappa\)-condensed objects
is useful for any~\(\kappa\).
We explain certain cases where it is useful.
In particular,
we see that
we can replace \(\injlim_{\kappa}\) in \cref{e:lwgc9}
with \(\bigcup_{\kappa}\)
by only considering
sufficiently large regular cardinals~\(\kappa\):

\begin{proposition}\label{xuf9uh}
Fix infinite regular cardinals~\(\aleph_{1}\leq\kappa\leq\lambda\).
  Suppose that
  \(\aleph_{1}\)-small limits commute with \(\kappa\)-filtered colimits
  in a presentable \(\infty\)-category~\(\cat{C}\).
  Then the functor
  \begin{equation*}
    \Shv^{\hyp}(\Cat{TDis}_{\kappa};\cat{C})
    \simeq
    \Cat{ConAni}_{\kappa}\otimes\cat{C}
    \to
    \Cat{ConAni}_{\lambda}\otimes\cat{C}
    \simeq
    \Shv^{\hyp}(\Cat{TDis}_{\lambda};\cat{C})
  \end{equation*}
  can be concretely
  described as the left Kan extension
  along
  \(\Cat{TDis}_{\kappa}\to\Cat{TDis}_{\lambda}\).
  The same holds
  when we replace \(\Cat{TDis}\) with \(\Cat{Cpt}\).
\end{proposition}

\begin{proof}
It suffices to show that
  the category
  \((\Cat{TDis}_{\kappa})_{S'/}\)
  is \(\kappa\)-cofiltered
  for \(S'\in\Cat{TDis}_{\lambda}\)
  and the same claim for \(\Cat{Cpt}\).
  But since \(\Cat{TDis}_{\kappa}\)
  and \(\Cat{Cpt}_{\kappa}\)
  are the full subcategories spanned by \(\kappa\)-cocompact objects,
  they have \(\kappa\)-small limits.
\end{proof}

\begin{corollary}\label{xj8jo5}
  Let \(\cat{C}\) be a presentable \(\infty\)-category.
  The \(\infty\)-category of condensed objects
  can be identified with
  the full subcategory of \(\Fun(\Cat{TDis}^{\op},\cat{C})\)
  or \(\Fun(\Cat{Cpt}^{\op},\cat{C})\)
  spanned by accessible functors
  satisfying the hypersheaf condition.
\end{corollary}

\begin{corollary}\label{xloozg}
  When \(X\) and~\(Y\) are condensed objects in
  a presentable \(\infty\)-category,
  \(\Map(Y,X)\) is small.
\end{corollary}

We then consider
describing condensed objects using
extremally disconnected sets,
which is more subtle:

\begin{proposition}\label{x9djsw}
Fix infinite strong limit cardinals~\(\kappa\leq\lambda\).
  Suppose that
  finite products commute with \(\cf(\kappa)\)-filtered colimits
  in a presentable \(\infty\)-category~\(\cat{C}\).
  Then the functor
  \begin{equation*}
    \Fun^{\times}(\Cat{EDis}_{\kappa}^{\op},\cat{C})
    \simeq
    \Cat{ConAni}_{\kappa}\otimes\cat{C}
    \to
    \Cat{ConAni}_{\lambda}\otimes\cat{C}
    \simeq
    \Fun^{\times}(\Cat{EDis}_{\lambda}^{\op},\cat{C})
  \end{equation*}
  can be concretely
  described as the left Kan extension
  along
  \(\Cat{EDis}_{\kappa}\to\Cat{EDis}_{\lambda}\).
  In particular,
  it is fully faithful.
\end{proposition}

\begin{proof}
  It suffices to show that
  the category \((\Cat{EDis}_{\kappa})_{S'/}\)
  is \(\cf(\kappa)\)-cofiltered
  for \(S'\in\Cat{EDis}_{\lambda}\).
  This was proven in the proof of~\cite[Proposition~2.9]{Condensed}.
\end{proof}

\begin{example}\label{xba8hv}
  Assume that
  finite products commute with filtered colimits
  in a presentable \(\infty\)-category~\(\cat{C}\);
  e.g., \(\cat{C}\) is either compactly generated or stable.
  Then by \cref{x9djsw} and \(\Cat{ConAni}_{\aleph_{0}}\simeq\Cat{Ani}\),
  the \(\infty\)-category~\(\cat{C}\) 
  fully faithfully sits inside
  the \(\infty\)-category of condensed objects.
  We come back to this point in \cref{ss:dis}.
\end{example}

\subsection{Conservativity of ultrafilters}\label{ss:kr}

Recall that
for an \(\infty\)-topos~\(\cat{X}\),
the \(\infty\)-category of points \(\Pt(\cat{X})\)
is the full subcategory of \(\Fun(\Cat{Ani},\cat{X})\)
spanned by left-exact colimit-preserving functors.
Concretely,
a point of \(\Cat{ConAni}_{\kappa}\)
is a \(\Pro\)-object~\(X\)
in \(\Cat{EDis}_{\kappa}\)
such that
\begin{equation*}
  \coprod_{i=1}^{n}
  \Map(P,S_{i})
  \to
  \Map(P,S_{1}\amalg\dotsb\amalg S_{n})
\end{equation*}
is surjective
for any extremally disconnected sets~\(S_{1}\), \dots,~\(S_{n}\).
Equivalently,
it is a \(\Pro\)-object~\(P\) such that
the maps
\begin{align}
  \label{e:1hxm5}
  \emptyset
  &\to
  \Map(P,\emptyset),&
  \Map(P,{*})\amalg\Map(P,{*'})
  &\to
  \Map(P,{*}\amalg{*'})
\end{align}
are equivalences.

\begin{example}\label{xyw599}
  Let \(X\) be an object of \(\Cat{EDis}_{\kappa}\)
  and \(x\in X\) be a point.
  Then we write \(X_{x}^{\dag}\)
  for the \(\Pro\)-object
  \((Z)_{Z}\)
  where \(Z\) runs over clopen subsets of~\(X\)
  containing~\(x\).
  It is easy to see that this is a point of \(\Cat{ConAni}_{\kappa}\)
  by checking that the maps \cref{e:1hxm5}
  are surjective.
\end{example}

Considering points of an \(\infty\)-topos is useful
because of Deligne's completeness theorem.
Here,
we see the following conservativity result:

\begin{theorem}\label{xzifdm}
  Fix an infinite strong limit cardinal~\(\kappa\).
  A map \(F\to F'\)
  between \(\kappa\)-condensed animas
  is an equivalence if and only if
  the map
  \begin{equation*}
    \injlim_{\mu(J)=1}F(\beta J)
    \to
    \injlim_{\mu(J)=1}F'(\beta J)
  \end{equation*}
  is an equivalence
  for any ultrafilter~\(\mu\) on
  any set~\(I\) of cardinality \(<\kappa\),
\end{theorem}

\begin{proof}
Since the ``only if'' direction is clear,
  we prove the ``if'' part.

  The equation
  is the evaluation of \(F\to F'\) at
  the \(\Pro\)-object \((\beta I)_{\mu}^{\dag}\).
  Note that
  for any point~\(x\) of an object~\(X\) of \(\Cat{EDis}_{\kappa}\),
  the \(\Pro\)-object \(X_{x}^{\dag}\) is a retract
  of \((\beta I)_{\mu}^{\dag}\)
  for some~\(I\) and~\(\mu\).
  By Deligne's completeness theorem,
  it suffices to show that
  any point is a cofiltered limit in \(\Pro(\Cat{EDis}_{\kappa})\)
  of points of the form~\(X_{x}^{\dag}\).

  We consider a point~\(P\).
  Let \((X_{j})_{j}\)
  be a codirected diagram in \(\Cat{EDis}_{\kappa}\)
  whose limit is~\(P\).
  For each~\(j\),
  there is a unique point~\(x_{j}\in X_{j}\) that is in the image
  of all transition maps to~\(X_{j}\):
  If there is no such point,
  \(P\) is empty and
  hence \(\emptyset\to\Map(P,\emptyset)\) is not surjective.
  If there are two such points,
  a map \(X_{j}\to{*}\amalg{*'}\)
  that separates those two
  cannot factor through~\(*\) or~\(*'\).
  Therefore,
  we obtain a map
  \begin{equation*}
    \projlim_{j}(X_{j})_{x_{j}}^{\dag}\to\projlim_{j}X_{j}=P.
  \end{equation*}
  We claim that this is an equivalence.
  We can construct an inverse.
  Let \(x_{j}\in Z\subset X_{j}\) be a clopen subset.
  We consider a function
  \(X_{j}\to{*}\amalg{*'}\)
  whose fiber at~\(*\) is exactly~\(Z\).
  From the fact that \(P\) is a point
  and the definition of~\(x_{j}\),
  we see that
  \(P\to X_{j}\to{*}\amalg{*'}\)
  should factor through~\(*\).
  Therefore,
  we have a map \(P\to Z\).
  By varying~\(j\) and~\(Z\),
  we get a map \(P\to\projlim_{j}(X_{j})_{x_{j}}^{\dag}\),
  which is the desired inverse.
\end{proof}

This result applies to more general coefficients:

\begin{corollary}\label{conservative}
Let \(\kappa\) be an infinite strong limit cardinal
  and \(\cat{C}\) a compactly generated \(\infty\)-category.
  A map \(C\to C'\) between condensed objects
  is an equivalence
  if and only if
  the morphism
  \begin{equation*}
    \injlim_{\mu(J)=1}C(\beta J)
    \to
    \injlim_{\mu(J)=1}C'(\beta J)
  \end{equation*}
  is an equivalence
  for any ultrafilter~\(\mu\) on
  any set~\(I\) of cardinality \(<\kappa\).
\end{corollary}

\begin{proof}
  Again, the ``only if'' direction is clear.
  We prove the ``if'' direction.
  Suppose that the condition is satisfied.
  Since \(\cat{C}\) is compactly generated,
  what we have to prove is that
  the map
  \(\Map(T,C(\X))\to\Map(T,C'(\X))\)
  of condensed animas
  for any compact object~\(T\) of~\(\cat{C}\),
  is an equivalence.
  Now we can conclude by applying \cref{xzifdm}.
\end{proof}

\begin{remark}\label{x16wx2}
  Since conservative functors
  are closed under retracts,
  \cref{conservative} remains true
  when \(\cat{C}\) is compactly assembled\footnote{This terminology is due to Lurie;
    see~\cite[Section~21.1.2]{LurieSAG}.
    Note that
    in the stable situation,
    this is equivalent to dualizability;
    see~\cite[Section~D.7]{LurieSAG}
  },
  i.e., a
  retract
  of a compactly generated \(\infty\)-category
  in \(\Cat{Pr}\).
\end{remark}

\subsection{Discrete animas}\label{ss:dis}

As we have seen in \cref{xba8hv},
in a nice situation,
it is harmless to consider an object as a condensed object:

\begin{definition}\label{x6p5hg}
  Assume that
  finite products commute with filtered colimits
  in a presentable \(\infty\)-category~\(\cat{C}\).
  We say that a condensed object in~\(\cat{C}\)
  is \emph{discrete} if it is \(\aleph_{0}\)-condensed.
\end{definition}

The following consequence
of our argument in \cref{ss:kr} is useful:

\begin{proposition}\label{dis}
Let \(\cat{C}\) be a compactly assembled \(\infty\)-category.
  A condensed object~\(C\) is discrete if and only if
  the map
  \begin{equation*}
    \injlim_{\mu(J)=1}
    C(\beta J)
    \to C(\{\mu\})
    \simeq C(*)
  \end{equation*}
  is an equivalence
  for any ultrafilter~\(\mu\) on any set~\(I\).
\end{proposition}

The ``only if'' direction
is strengthened in \cref{xd9foy}.

\begin{proof}
  The ``only if'' direction
  follows from the fact
  that finite products commute with filtered colimits
  in~\(\cat{C}\).
  We prove the ``if'' direction.

  Suppose that the condensed object~\(C\)
  satisfies the condition.
  We write \(C'\) for the discrete condensed object
  corresponding to the object~\(C(*)\in\cat{C}\).
  We have a canonical map \(C'\to C\).
  From the condition and the ``only if''
  direction,
  \begin{equation*}
    \injlim_{\mu(J)=1}
    C'(\beta J)
    \to
    \injlim_{\mu(J)=1}
    C(\beta J)
  \end{equation*}
  is an equivalence for any ultrafilter~\(\mu\)
  on any set~\(I\).
  Therefore, by \cref{x16wx2},
  \(C'\to C\) is an equivalence
  and hence \(C\) is discrete.
\end{proof}

\begin{definition}\label{xjcm9x}
  Let \(C\) be a condensed object
  in a presentable \(\infty\)-category
  in which finite products commute with filtered colimits.
  We say that
  a morphism \(C\to C'\) (or simply~\(C'\)) is its \emph{discretization}
  if \(C'\) is discrete
  and it is initial among such morphisms.
  By definition,
  discretization is functorial
  when we consider the full subcategory
  of discretizable condensed objects.
  We write \((\X)^{\dis}\) for this (partial) functor.
\end{definition}

\begin{example}\label{dis_colim}
In the situation of \cref{xjcm9x},
  let \(C=\injlim_{i}C_{i}\) be
  a colimit of discretizable objects.
  Then \(C\) is discretizable
  and the map
  \(\injlim_{i}(C_{i})^{\dis}\to C^{\dis}\)
  is an equivalence.
\end{example}

\begin{example}\label{xh0gyx}
  Consider a profinite set \(X=\projlim_{i}X_{i}\)
  as a condensed anima.
  Its discretization does not exist
  unless \(X\) is finite:
  Suppose that the discretization exists.
  By the universality
  and the fact that \(X_{i}\) is discrete for each~\(i\),
  the identity map
  \(X\to X=\projlim_{i}X_{i}\)
  factors through \(X\to X^{\dis}\).
  Then \(X\) is a retract of~\(X^{\dis}\)
  and therefore \(X\) itself must be discrete.
  See also \cref{xs7llv} about this example.
\end{example}

Discretization is meaningful in a more geometric situation;
we see more computations in \cref{ss:core}.

Lastly,
we consider the value of a discrete condensed object
at profinite sets:

\begin{proposition}\label{dis_pf}
  Let \(\cat{C}\) be a compactly assembled \(\infty\)-category.
  Let \(F\colon\Cat{TDis}^{\op}\to\cat{C}\)
  be a condensed object (via \cref{xj8jo5})
  that is discrete.
  Then \(\injlim_{i}F(S_{i})\to F(S)\)
  is an equivalence
  for profinite sets \(S=\projlim_{i}S_{i}\).
\end{proposition}

\begin{proof}
  By \cref{xuf9uh},
  it suffices to prove that
  \(\cat{C}\simeq\Shv^{\hyp}(\Cat{TDis}_{\aleph_{1}};\cat{C})
  \to\Shv^{\hyp}(\Cat{TDis}_{\aleph_{1}};\cat{C})\)
  can be computed by left Kan extension.
  In other words,
  we have to prove that
  for any object~\(C\),
  the functor
  \(G\colon\Cat{TDis}_{\aleph_{1}}^{\op}\to\cat{C}\)
  defined by the formula
  \(G(S)=\injlim_{i}C^{\times S_{i}}\) is a hypersheaf.

  We choose a compactly generated \(\infty\)-category~\(\cat{C}'\)
  and a localization \(\cat{C}'\to\cat{C}\)
  that preserves both limits and colimits.
  In this case, the functor
  \begin{equation*}
    \Shv^{\hyp}(\Cat{TDis}_{\aleph_{1}};\cat{C}')
    \simeq
    \Cat{ConAni}_{\aleph_{1}}\otimes\cat{C}'
    \to
    \Cat{ConAni}_{\aleph_{1}}\otimes\cat{C}
    \simeq
    \Shv^{\hyp}(\Cat{TDis}_{\aleph_{1}};\cat{C})
  \end{equation*}
  is given by composition
  so we are reduced to the compactly generated case.
  Then by Yoneda (cf.~the proof of \cref{conservative}),
  we are reduced to the case of \(\cat{C}=\Cat{Ani}\).
  In that case,
  by writing an anima as the limit
  of its Postnikov tower,
  we can reduce to the case of \(\cat{C}=\Cat{Ani}_{\leq n}\).
  In that case,
  hyperdescent is equivalent to Čech descent.
  Then we can reduce the descent for
  surjections between profinite sets
  to the descent for surjections between finite sets,
  which split.
\end{proof}

\begin{corollary}\label{xd9foy}
  Let \(\cat{C}\) be a compactly assembled \(\infty\)-category.
  For a discrete condensed object~\(C\)
  and a point~\(x\) on a profinite set~\(X\),
  the map
  \begin{equation*}
    \injlim_{x\in Z\subset X}C(Z)
    \to
    C(x)
    \simeq
    C(*)
  \end{equation*}
  is an equivalence,
  where \(Z\) runs over the clopen neighborhoods of~\(x\).
\end{corollary}

\subsection{Recollection: solid spectra}\label{ss:sol}

In this section,
we review some basic facts about solid spectra.

\begin{theorem}[Clausen--Scholze]\label{xdjyu0}
  Let \(\Cat{SolSp}_{\geq0}\subset\Cat{ConSp}_{\geq0}\)
  be the full subcategory
  generated by products of~\(\SS\)
  under colimits and extensions.
  \begin{enumerate}
    \item
      An object \(E\in\Cat{ConSp}_{\geq0}\)
      is inside \(\Cat{SolSp}_{\geq0}\)
      if and only if
      \(\pi_{n}E\) is inside \(\Cat{SolSp}_{\geq0}\)
      for \(n\geq0\).
    \item
      For any profinite set~\(S\)
      and \(N\in\Cat{SolSp}_{\geq0}\),
      the internal mapping condensed spectrum
      \([\SS[S],N]\) is in \(\Cat{SolSp}_{\geq0}\).
    \item
      The inclusion
      \(\Cat{SolSp}_{\geq0}\subset\Cat{ConSp}_{\geq0}\)
      has a left adjoint.
    \item
      For a profinite set~\(S\),
      the image of \(\SS[S]\) under the left adjoint
      is a product of~\(\SS\).
  \end{enumerate}
\end{theorem}

To prove this, we need the following:

\begin{lemma}\label{xui778}
  For a set~\(I\),
  the canonical map
  \(\SS^{I}\otimes_{\SS}\ZZ\to\ZZ^{I}\)
  of condensed spectra is an equivalence.
\end{lemma}

\begin{proof}
Let \(E\) be a (discrete) spectrum
  and we consider a general question
  whether \(\SS^{I}\otimes_{\SS}E\to E^{I}\)
  is an equivalence.
  We prove that this is the case when \(E\) is pseudocoherent.

  First, when \(E\) is perfect, the morphism is an equivalence.
  When \(E\) is pseudocoherent, we write \(E\) as a sequential colimit of perfect spectra
  \(\injlim_{n\geq0}E_{n}\)
  such that \(\cofib(E_{n}\to E)\) is \((n+1)\)-connective.
  We consider the diagram
  \begin{equation*}
    \begin{tikzcd}
      \SS^{I}\otimes E_{n}\ar[r,"\simeq"]\ar[d]&
      E_{n}^{I}\ar[d]\\
      \SS^{I}\otimes E\ar[r]&
      E^{I}\rlap.
    \end{tikzcd}
  \end{equation*}
  Then the cofibers of the vertical arrows are \((n+1)\)-connective
  and hence so is its total cofiber.
  Therefore, the cofiber of the bottom arrow
  is \((n+1)\)-connective.
  By varying~\(n\),
  we see that the bottom arrow is an equivalence.
\end{proof}

\begin{proof}[Proof of \cref{xdjyu0}]
  By applying the nilinvariance theorem
  proven in~\cite[Proposition~12.23]{Analytic}
  to the solid structure on~\(\ZZ\),
  which was introduced in~\cite[Sections~5--6]{Condensed},
  we get the solid structure on~\(\SS\).
  We wish to check that
  the \(\infty\)-category of connective solid spectra
  coincides with \(\Cat{SolSp}_{\geq0}\)
  in the statement.
  To prove that,
  it suffices to check that
  \(\SS[S]^{\sol}\) is a product of~\(\SS\)
  for a profinite set~\(S\).

  We pick a \(\ZZ\)-linear map
  \(\ZZ[S]\to\ZZ^{I}\)
  that induces an isomorphism after (\(\ZZ\)-)solidification.
  We can lift this to 
  \(\SS[S]\to\SS^{I}\).
  Since \(\SS^{I}\) is solid,
  we obtain the map \(\SS[S]^{\sol}\to\SS^{I}\).
  This is an equivalence after \(\X\otimes_{\SS}\ZZ\)
  by the definition of the solid structure on~\(\SS\)
  and \cref{xui778}.
  Therefore,
  we have \(\SS[S]^{\sol}\simeq\SS^{I}\).
\end{proof}

\begin{definition}\label{x8osa7}
  Using the notation in \cref{xdjyu0},
  we call an object of \(\Cat{SolSp}\),
  the stabilization of \(\Cat{SolSp}_{\geq0}\),
  a \emph{solid spectrum}.
  We write \((\X)^{\sol}\)
  for the left adjoint to the inclusion \(\Cat{SolSp}\hookrightarrow\Cat{ConSp}\)
  and call it \emph{solidification}.
\end{definition}

We note the following:

\begin{example}\label{real_sol}
  The proof of the existence of solid theory
  shows that
  the real numbers~\(\RR\) (with the usual topology)
  solidify to zero;
  see~\cite[Corollary~6.1]{Condensed}.
  Therefore,
  any condensed \(\RR\)-module (spectrum) solidifies to zero.
\end{example}

We see more computations of solidification
in \cref{ss:core}.

\section{Truncatedness of solidified \texorpdfstring{\(K\)}{K}-theory}\label{s:tr}

In this section,
we see the following truncatedness property:

\begin{theorem}\label{tr}
  For a connective condensed \(\E_1\)-algebra~\(A\)
  whose underlying condensed spectrum
  has a structure of an \(\RR\)-module,
  the map \(K(A)\to K(\pi_0A)\)
  solidifies to an equivalence.
\end{theorem}

We prove this in \cref{ss:tr}
and see consequences in \cref{ss:lt}.

\begin{remark}\label{x5cbuy}
  Konovalov~\cite{Konovalov22}
  proved the truncatedness result
  of semitopological \(K\)-theory.
\end{remark}

\begin{remark}\label{x20z5b}
The same argument shows
  that for connective condensed \(\E_{1}\)-rings
  underlying \(\QQ_{p}\)-modules,
  the \(\SS[1/p]\)-solidification\footnote{The solid structure on~\(\SS[1/p]\)
    is obtained from that on~\(\ZZ[1/p]\)
    constructed in~\cite[Theorem~8.1]{Condensed}
    via the nilinvariance theorem
    proven in~\cite[Proposition~12.23]{Analytic}.
  } of~\(K\)
  is truncating.
  The key input is the fact that the \(\SS[1/p]\)-solidification
  (or equivalently, \(\ZZ[1/p]\)-solidification)
  of~\(\QQ_{p}\) vanishes.
\end{remark}

\subsection{Truncatedness}\label{ss:tr}

The main ingredient is the following:

\begin{proposition}\label{89d1a5228f}
  Let \(A\) be a condensed \(\E_{1}\)-algebra
  whose underlying condensed spectrum
  has a structure of an \(\RR\)-module.
  Then the solidification of
  its cyclic homology
  \(\HC(A/\QQ)\) is zero.
\end{proposition}

\begin{proof}
  Note that
  the kernel of
  the solidification functor \(\Cat{ConSp}\to\Cat{SolSp}\)
  is closed under colimits.
  Since \(\HC(A/\QQ)\)
  can be written as a colimit
  of objects admitting \(\RR\)-actions,
  the desired result follows from \cref{real_sol}.
\end{proof}

\begin{proof}[Proof of~\cref{tr}]
  According to Goodwillie's theorem~\cite{Goodwillie86},
  the fibers of
  \(K(A)\to K(\pi_0A)\)
  and
  \(\HC(A/\QQ)\to\HC(\pi_0A/\QQ)\)
  are equivalent up to a shift.
  The latter morphism
  solidifies to an equivalence
  by \cref{89d1a5228f},
  hence so does the former.
\end{proof}

\subsection{Consequences}\label{ss:lt}

Land--Tamme~\cite[Section~3]{LandTamme19} showed that
the truncatedness property of a localizing invariant has many useful consequences.
We use their argument to show
the excision and nilinvariance properties:

\begin{corollary}\label{exc}
  Let
  \begin{equation*}
    \begin{tikzcd}
      A\ar[r]\ar[d]&
      A'\ar[d]\\
      B\ar[r]&
      B'
    \end{tikzcd}
  \end{equation*}
  be a cartesian square of condensed \(\E_1\)-rings
  such that
  every term is connective
  and
  the canonical map \(\pi_0(B\otimes_{A}A')\to\pi_0(B')\)
  is an equivalence.
  If the underlying condensed spectra
  of \(A'\) and~\(B'\)
  have \(\RR\)-module structures,
  \(K(\X)^{\sol}\) carries
  the square to a cartesian square
  of condensed spectra.
\end{corollary}

\begin{proof}
  We follow the proof of~\cite[Theorem~3.3]{LandTamme19}.
  Let \(C=B\odot_A^{B'}A'\to B'\)
  be (the condensed version of) the Land--Tamme construction.
  By~\cite[Theorem~1.3]{LandTamme19},
  it suffices to show that
  \(K(C)\to K(B')\) solidifies to an equivalence.
  Since the underlying condensed spectrum of~\(C\)
  is equivalent to \(B\otimes_{A}A'\),
  it has an \(\RR\)-module structure.
  Therefore,
  we obtain the desired result
  by \cref{tr}.
\end{proof}

\begin{corollary}\label{nil}
  Let \(A\to B\) be a surjection of static condensed rings
  over~\(\RR\) whose kernel is nilpotent.
  Then \(K(A)\to K(B)\) solidifies to an equivalence.
\end{corollary}

\begin{proof}
  The proof of~\cite[Corollary~3.5]{LandTamme19}
  works here as well.
\end{proof}

\section{Discreteness of solidified \texorpdfstring{\(K\)}{K}-theory}\label{s:dis}

By solidifying \(K\) of a condensed ring,
we still get a condensed spectrum instead of a (discrete) spectrum.
In this section,
we observe that
it is automatically discrete in certain cases:

\begin{theorem}\label{dis_ban}
  Let \(A\) be a compactly generated Hausdorff
  real topological algebra\footnote{We do not impose local convexity here;
    see also \cref{xv4il4}.
  }
  such that \(\GL_{r}(A)\subset\Mat_{r}(A)\) is open
  for \(r\geq0\)\footnote{When \(A\) is commutative,
    it suffices to check this for \(r=1\).
  }.
  Then \((K^{\cn}(A))^{\sol}\)
  is discrete.
\end{theorem}

\begin{theorem}\label{dis_alg}
  Let \(A\) be
  a base change of
  a connective \(\E_{1}\)-algebra over~\(\RR^{\delta}\)
  along \(\RR^{\delta}\to\RR\).
  Then \(K(A)^{\sol}\) is discrete.
\end{theorem}

\begin{remark}\label{xxqbww}
  Real Banach algebras satisfy the condition of \cref{dis_ban},
  but general real Fréchet algebras do not:
  Consider the ring of entire functions \(\Cls{O}(\CC)\)
  with the topology generated by the seminorms
  \(f\mapsto\sup_{\lvert z\rvert\leq n}\lvert f(z)\rvert\)
  for \(n\geq0\).
  Then any neighborhood of the constant function~\(1\)
  contains a function that vanishes somewhere.
\end{remark}

We deduce \cref{dis_ban,dis_alg} in \cref{ss:yes} from two ingredients:
a generalization of the Breen--Deligne resolution, which we construct in \cref{ss:bd},
and a further study
on the solidification of a free condensed spectrum in \cref{ss:useful}.

\subsection{Group completion via the Breen--Deligne resolution}\label{ss:bd}

We write \(\Cat{CMon}\) for the \(\infty\)-category
of \(\E_{\infty}\)-monoids
and \(\Cat{CGrp}\) for the subcategory
spanned by grouplike \(\E_{\infty}\)-monoids.
The inclusion has a left adjoint
\((\X)^{\gp}\colon\Cat{CMon}\to\Cat{CGrp}\).
We prove the following:

\begin{theorem}\label{bd}
  There is a functor
  \begin{equation*}
    \Cat{CMon}
    \xrightarrow{\BD}
    \Fun(\Cat{\Delta}^{\op},\Cat{Sp}_{\geq0})
  \end{equation*}
  satisfying the following:
  \begin{itemize}
    \item
      There is an equivalence
      \(\lvert\BD(\X)_{\bullet}\rvert\simeq(\X)^{\gp}\)
      of functors from \(\Cat{CMon}\)
      to \(\Cat{Sp}_{\geq0}\simeq\Cat{CGrp}\).
    \item
      For each \(\bullet\),
      the functor
      \(\BD(\X)_{\bullet}\colon\Cat{CMon}\to\Cat{Sp}_{\geq0}\)
      is a finite direct sum
      of functors of the form \(\SS[(\X)^{n}]\)
      for \(n\geq0\).
  \end{itemize}
\end{theorem}

\begin{definition}\label{fd6bfc3942}
  We fix a functor~\(\BD\)
  satisfying \cref{bd}
  throughout this paper
  and call it the \emph{Breen--Deligne resolution}.
\end{definition}

\begin{example}\label{3c6790008c}
  By restricting~\(\BD\) to \(\Cat{Ab}\subset\Cat{CMon}\)
  and composing it with \(\ZZ\otimes_{\SS}\X\),
  we get a classical Breen--Deligne resolution
  (cf.~\cite[Theorem~4.5]{Condensed})
  in light of the Dold--Kan correspondence.
\end{example}

\begin{example}\label{e7d4d403fa}
  Let \(\cat{X}\) be an arbitrary \(\infty\)-topos.
  Then by tensoring~\(\BD\),
  we get the functor
  \begin{equation*}
    \CMon(\cat{X})
    \to
    \Fun(\Cat{\Delta}^{\op},\Sp(\cat{X})_{\geq0})
    \to
    \Sp(\cat{X})_{\geq0}.
  \end{equation*}
  We also call this functor the Breen--Deligne resolution.
  In particular,
  if we apply this to \(\cat{X}=\Cat{ConAni}_{\kappa}\)
  and vary~\(\kappa\),
  we get its condensed version.
\end{example}

We show \cref{bd}
using the pseudocoherence argument
as in~\cite[Section~4.1]{ClausenMathewMorrow21}.
The key input is the following variant of~\cite[Lemma~4.24]{ClausenMathewMorrow21}:

\begin{lemma}\label{x252fc}
  Let \(\cat{C}\) denote the full subcategory
  of \(\Cat{CMon}\) spanned by
  the free commutative monoids on finite sets.
  The functor
  \begin{equation*}
    F\colon
    \cat{C}\hookrightarrow\Cat{CMon}
    \to\Cat{CGrp}\simeq\Cat{Sp}_{\geq0}
  \end{equation*}
  determined by group completion is pseudocoherent
  in \(\Fun(\cat{C},\Cat{Sp}_{\geq0})\);
  i.e.,
  for any \(i\),
  there is a compact object~\(G\)
  and an equivalence \(\tau_{\leq i}F\simeq\tau_{\leq i}G\).
\end{lemma}

\begin{proof}
  Concretely,
  we have
  \begin{equation*}
    F(\X)=
    \injlim_n
    \Sigma^{-n}\Sigma^{\infty}(B^n(\X)).
  \end{equation*}
  By the Freudenthal suspension theorem,
  it suffices to show that
  \(\Sigma^{\infty}(B^{n}(\X))\)
  is pseudocoherent for each~\(n\).
  Since \(BM\) can be computed as
  a geometric realization
  of \(M^{\bullet}\),
  we can compute
  \(B^{n}M\) as a geometric realization
  of a multisimplicial object
  whose terms are finite products of copies of~\(M\).
  The desired result follows from this observation.
\end{proof}

\begin{proof}[Proof of \cref{bd}]
  We use the notation of \cref{x252fc}.
  By \cref{x252fc},
  according to \(\text{(1)}\Rightarrow\text{(3)}\)
  of~\cite[Lemma~4.3]{ClausenMathewMorrow21},
  we can write~\(F\) as a colimit of
  \begin{equation*}
    0=
    D_{-1}\to
    D_{0}\to
    \dotsb
  \end{equation*}
  such that \(\Sigma^{-n}\cofib(D_{n-1}\to D_{n})\)
  is a finite direct sum of representables.
By applying~\cite[Theorem~1.2.4.1]{LurieHA}
  to this diagram,
  we obtain a functor
  \(\cat{C}\to\Fun(\Cat{\Delta}^{\op},\Cat{Sp})\),
  whose left Kan extension
  satisfies the desired conditions
  by~\cite[Remark~1.2.4.7]{LurieHA}.
\end{proof}

\subsection{Digression: discretization and shape}\label{ss:core}

By \cref{bd},
to understand the solidification of a condensed \(\E_{\infty}\)-group~\(M^{\gp}\),
we need to understand the solidification of \(\SS[M^{n}]\).
Our goal in this section is to get a sufficient condition
for \(\SS[X]^{\sol}\),
or equivalently \(\ZZ[X]^{\sol}\), to be discrete
for a compactum~\(X\).
We approach this by
considering a similar question on discretization,
which we have a complete answer in terms of sheaf theory.

We first introduce general machinery.
Recall that discretization,
which is the left adjoint
to the inclusion \(\Cat{Ani}\hookrightarrow\Cat{ConAni}\),
does not exist (see \cref{xh0gyx}) but
it exists after \(\Pro\)-extending the target.
This was considered by Mair
in~\cite[Section~7.2]{Mair}
up to size issues,
which we address here:

\begin{proposition}\label{xji2wv}
  Let \(Y\) be a condensed anima.
  Then the functor
  \(\Cat{Ani}\to\Cat{Ani}\)
  given by
  \(X\mapsto\Map(Y,X)\)
  is left exact and accessible.
\end{proposition}

\begin{proof}
  We fix an infinite strong limit cardinal~\(\kappa\)
  such that \(Y\) is \(\kappa\)-condensed.
  Since
  the inclusion \(\Cat{Ani}\hookrightarrow\Cat{ConAni}_{\kappa}\)
  is left exact and accessible by \cref{x9djsw},
  the desired result follows.
\end{proof}

\begin{definition}\label{x9u2yo}
  We write
  \begin{equation*}
    (\X)^{\Prodis}\colon
    \Cat{ConAni}\to\Pro(\Cat{Ani})
  \end{equation*}
  for the functor
  given in \cref{xji2wv}
  and call it the \emph{\(\Pro\)-discretization} functor.
\end{definition}

\begin{example}\label{x1l4m4}
  When a condensed anima~\(X\)
  admits a discretization,
  the \(\Pro\)-discretization
  \(X^{\Prodis}\) is constant
  and is equivalent to~\(X^{\dis}\).
\end{example}

\begin{theorem}\label{xyrqub}
  Let \(X\) be a compactum
  regarded as a condensed anima.
  Then the shape (see~\cite[Definition~7.1.6.3]{LurieHTT} for the definition)
  of the Postnikov-completed sheaf \(\infty\)-topos~\(\Shv^{\Post}(X)\)
  is canonically equivalent
  to the \(\Pro\)-discretization of~\(X\).
\end{theorem}

Typical applications are as follows:

\begin{example}\label{xs7llv}
  Consider
  a profinite set \(X=\projlim_{i}X_{i}\)
  as a condensed anima.
  Then \(X^{\Prodis}\)
  is canonically equivalent to \(\projlim_{i}X_{i}\) as a \(\Pro\)-anima
  by \cref{xyrqub}.
  This follows from
  the fact that \(\Shv(X)\) is Postnikov complete
  and the fact that its shape can be computed as
  \(\projlim_{i}\Sh(\Shv(X_{i}))\)
  by~\cite[Proposition~2.11]{Hoyois18}.
\end{example}

\begin{example}\label{cw_dis}
  Let \(X\) be the underlying condensed set
  of a finite CW~complex.
  Then its discretization (as a condensed anima)
  exists and is equivalent to its homotopy type
  by \cref{xyrqub}.
\end{example}

Our proof of \cref{xyrqub}
uses a certain nonabelian generalization
of~\cite[Theorem~3.2]{Condensed},
which was attributed to Dyckhoff there.
The following was proven
by Haine in~\cite[Proposition~0.8]{Haine}
up to size issues,
which we address here:

\begin{lemma}\label{xiajq0}
  Let \(\kappa\) be an infinite strong limit cardinal.
  For \(X\in\Cat{Cpt}_{\kappa}\),
  there is a functorial construction
  of a geometric morphism
  \((\Cat{ConAni}_{\kappa})_{/X}\to\Shv^{\Post}(X)\)
  whose left adjoint is fully faithful.
\end{lemma}

\begin{proof}
  We follow the argument by Haine~\cite{Haine};
  we check that it is compatible with our cardinality bound.

  Let \(\Cat{Top}\) be
  the \(\infty\)-category of \(\infty\)-toposes
  with geometric morphisms.
  We first see that
  both \((\Cat{ConAni}_{\kappa})_{/\X}\) and \(\Shv^{\Post}(\X)\)
  are hypersheaves on~\(\Cat{Cpt}_{\kappa}\) valued in \(\Cat{Top}\):
The first one is a hypersheaf
  by descent
  and
  the second one is a hypersheaf
  by~\cite[Theorem~0.5]{Haine}.
  Therefore, it suffices to construct
  such a geometric morphism for \(X\in\Cat{EDis}_{\kappa}\).
  In this case, the argument of~\cite[Corollary~4.4]{Haine}
  works since closed subsets of~\(X\)
  are \(\kappa\)-condensed.
\end{proof}

\begin{proof}[Proof of \cref{xyrqub}]
We consider the commutative square
  \begin{equation*}
    \begin{tikzcd}
      (\Cat{ConAni}_{\kappa})_{/X}\ar[r,"f"]\ar[d,"e"']&
      \Shv^{\Post}(X)\ar[d,"p"]\\
      \Cat{ConAni}_{\kappa}\ar[r,"q"]&
      \Cat{Ani}
    \end{tikzcd}
  \end{equation*}
  in \(\Cat{Top}\),
  where \(f\) is the morphism in \cref{xiajq0},
  \(e\) is the étale morphism map~\(X\),
  and \(p\) and~\(q\) are tautological.
  Let \(E\) be an anima.
  Then
  \begin{equation*}
    \Map(X,q^{*}E)
    \simeq
    \Map(e_{!}{*},q^{*}E)
    \simeq
    \Map({*},e^{*}q^{*}E)
    \simeq
    \Map(f^{*}{*},f^{*}p^{*}E).
  \end{equation*}
  Then by \cref{xiajq0},
  we can identify this as
  \(\Map({*},p^{*}E)\).
  Therefore, the desired claim follows.
\end{proof}

We later need the following strengthened form of \cref{cw_dis}:

\begin{corollary}\label{cw_indis}
  Let \(Y\) be the underlying condensed set
  of a finite CW~complex.
  Then for any discrete condensed anima~\(X\),
  the map
  \([Y^{\dis},X]\to[Y,X]\) of condensed animas
  (here \(Y^{\dis}\) exists by \cref{cw_dis})
  is an equivalence.
\end{corollary}

\begin{proof}
  It suffices to prove that
  \((Y\times S)^{\Prodis}\to(Y^{\dis}\times S)^{\Prodis}\)
  is an equivalence of \(\Pro\)-animas.
This follows from \cref{xcua8m} below
and the fact that
  finite colimits commute with cofiltered limits in
  \(\Pro(\Cat{Ani})\).
\end{proof}

\begin{lemma}\label{xcua8m}
  Let \(X\) and~\(X'\) be compacta.
  When each of \(X\), \(X'\), and \(X\times X'\)
  has a Postnikov-complete sheaf \(\infty\)-topos,
  the map
  \((X\times X')^{\Prodis}
  \to X^{\Prodis}\times X'^{\Prodis}\)
  is an equivalence.
\end{lemma}

\begin{proof}
  This follows from \cref{xyrqub}
  and~\cite[Proposition~2.11]{Hoyois18}.
\end{proof}

We then study solidification:

\begin{proposition}\label{xm65jj}
  Let \(X\) be a compactum
  having a discretization
  such that \(\SS[X^{\dis}]\simeq\SS[X]^{\dis}\)
is a pseudocoherent connective spectrum.
  Then \(\SS[X]\to\SS[X^{\dis}]\) is the solidification map.
\end{proposition}

\begin{proof}
  It suffices to show that
  \(\ZZ[X]\to\ZZ[X^{\dis}]\) is the solidification morphism.
  Since \(\ZZ[X^{\dis}]\) is discrete,
  we have a map \(\ZZ[X]^{\sol}\to\ZZ[X^{\dis}]\).
  We wish to show that this is an equivalence.
  Since
  the source and target are pseudocoherent connective solid spectra,
we can check this after applying \([\X,\ZZ]_{\ZZ}\),
  but in that case,
  the result follows from the definition of discretization.
\end{proof}

\begin{example}\label{x9jpqc}
  The underlying condensed set of a finite CW~complex
  satisfies the hypothesis of \cref{xm65jj}
  by \cref{cw_dis}.
  This was already covered in~\cite[Example~6.5]{Condensed}.
\end{example}

We have the following new example:

\begin{example}\label{warsaw}
  We have \(\SS[W]^{\sol}\simeq\SS\oplus\Sigma\SS\)
  when \(W\) is the Warsaw circle.
  This follows from
  the fact that \(\Shv(W)\) is Postnikov complete
  and \cref{xyrqub,xm65jj}.
\end{example}

Lastly,
as an aside,
we explain how to treat
infinite-dimensional compacta
using condensed mathematics:

\begin{remark}\label{xvgolx}
  Suppose that we want to treat
  invariants attached to
  the shape of~\(\Shv(X)\)
  instead of \(\Shv^{\Post}(X)\).
  \Cref{xyrqub} shows a limitation
  of the usefulness of condensed mathematics
  in that situation.
  However,
  we can compensate this in some cases.

  We write \(\Cat{Cpt}_{\fin}\) for
  the full subcategory of \(\Cat{Cpt}\)
  spanned by finite-dimensional compacta of countable weight.
  Note that \(\Shv(X)\)
  for \(X\in\Cat{Cpt}_{\fin}\)
  is Postnikov complete.
  Note that
  any functor
  \(\Cat{Cpt}^{\op}\to\cat{D}\)
  to a presentable \(\infty\)-category
  that preserves filtered colimits
is a left Kan extension
  of \(\Cat{Cpt}_{\fin}^{\op}\to\cat{D}\).

  As a concrete example,
  we here construct
  the algebraic-to-topological comparison map
  \(K^{\cn}(\Cls{C}(X;\HH))\to\Gamma(X;\cst{\ksp})\)
  of spectra\footnote{We can change the source to
    \(K(\Cls{C}(X;\HH))\)
    by using the fact that the condensed spectrum~\(K(\HH)\)
    is connective,
    which follows from~\cite[Theorem~A]{k-ros-1}.
  } natural in~\(X\)
  even though \(\Gamma(\Shv^{\Post}(X);\cst{\ksp})\)
  in general differs from the target.
  By the observation above,
  we only need to construct this for \(X\in\Cat{Cpt}_{\fin}\).
  As a condensed spectrum,
  we have a map
  \(K^{\cn}(\HH)\to{\ksp}\),
  e.g., by \cref{com_ban}.
  By \cref{xyrqub},
  we obtain the desired map
  as the composite
  \begin{equation*}
    K^{\cn}(\Cls{C}(X;\HH))
    \to
    K^{\cn}(\HH)(X)
    \to
    {\ksp}(X)
    \simeq
    \Gamma(X;\cst{\ksp})
  \end{equation*}
  for \(X\in\Cat{Cpt}_{\fin}\).
\end{remark}

\subsection{Locally contractible spaces}\label{ss:useful}

In \cref{ss:core},
we have addressed the question of when the solidification of \(\SS[X]\) is discrete
for a compactum~\(X\).
The proof of our main discreteness results needs the situation where \(X\) is not a compactum.
We prove the version we use in \cref{lslc}
by reducing it to our considerations in \cref{ss:core} for \(X = [0,1]\).

\begin{definition}\label{xv78wj}
  A condensed anima is \emph{contractible}
  if there is a point \(x\colon{*}\to X\)
  and a homotopy \(h\colon X\times[0,1]\to X\)
  that connects~\(x\) and~\(\id_{X}\).
\end{definition}

\begin{proposition}\label{contract}
  Let \(X\) be a contractible condensed anima.
  \begin{enumerate}
    \item\label{i:c_dis}
      The tautological map
      \(X\to{*}\) is the discretization map.
    \item\label{i:c_sol}
      The tautological map
      \(\SS[X]\to\SS[*]\simeq\SS\) is the solidification map.
  \end{enumerate}
\end{proposition}

\begin{proof}
We consider
  \begin{equation*}
    X
    \rightrightarrows
    X\times[0,1]
    \xrightarrow{h}
    X,
  \end{equation*}
  where the parallel arrows
  are determined by~\(0\) and~\(1\in[0,1]\).
  Now \cref{i:c_dis} follows from
  the fact that
  \(\Map(\X,E)\)
  for any discrete anima~\(E\)
  carries the parallel arrows to homotopic morphisms,
  which follows from \cref{cw_indis}.
  Similarly, \cref{i:c_sol} follows from
  the fact that
  \(\SS[\X]^{\sol}\) makes them homotopic,
  which follows from \cref{x9jpqc}.
\end{proof}

\begin{example}\label{xffpcl}
  We consider the Hilbert cube \(Q=[0,1]^{\NN}\),
  which is contractible.
  In this case, \(\Shv(Q)\) is not hypercomplete;
  see~\cite[Counterexample~6.5.4.2]{LurieHTT}.
  By \cref{contract},
  it still has a discretization.
  Hence by \cref{xyrqub},
  we see that
  the shape of \(\Shv^{\Post}(Q)\) is trivial.
\end{example}

We then consider
compactly generated Hausdorff spaces
that satisfy a certain local contractibility condition.
We start with a few general remarks:

\begin{remark}\label{x7o7di}
The inclusion
  of the category of compactly generated Hausdorff spaces
  into that of all Hausdorff spaces
  does \emph{not} preserve limits (even binary products).
  In any case,
  if we compose it with
  the functor
  to \(\Cat{ConSet}\),
  it preserves limits.
\end{remark}

\begin{remark}\label{xmyz9h}
  A compactly generated Hausdorff space~\(X\)
  is contractible as a condensed anima
  if and only if it is contractible in the usual sense;
  in spite of \cref{x7o7di},
  the product \(X\times[0,1]\)
  taken in the category of topological spaces
  is also the product in \(\Cat{ConAni}\).
\end{remark}

\begin{theorem}\label{lslc}
  Let \(X\) be a compactly generated Hausdorff space.
  Suppose that
for any open subset~\(U\subset X\)
  and a point~\(x\in U\),
  there is a neighborhood~\(V\)
  such that \(V\to U\) factors through
  a contractible condensed set.
  Then \(X\) admits a discretization
  as a condensed anima
  and \(\SS[X]\to\SS[X^{\dis}]\)
  is the solidification map.
\end{theorem}

\begin{lemma}\label{hcov}
  An open (semisimplicial) hypercover of
  a compactly generated Hausdorff space
  induces a colimit diagram
  in \(\Cat{ConAni}\).
\end{lemma}

\begin{proof}
Note that any open subset
  of a compactly generated Hausdorff space
  is again compactly generated Hausdorff.
  Hence,
  it suffices to show
  that any open cover
  \(\{U_{i}\to X\}_{i\in I}\)
  of a compactly generated Hausdorff space~\(X\)
  is jointly surjective as a family of maps of condensed sets.
  By base change,
  we can assume that \(X\) is a compactum
  and also \(I\) is finite.
  In that case,
  we can shrink each \(U_{i}\)
  to get a finite closed cover.
\end{proof}

\begin{proof}[Proof of \cref{lslc}]
By our assumption,
  we can construct an open hypercover~\(X_{\bullet}\)
  of~\(X\)
  that is interpolated as
  \begin{equation*}
    \begin{tikzcd}
      \cdots\ar[d]&
      X_{1}\ar[d]&
      X_{0}\ar[d]&
      X\rlap,\\
      X_{2}'\ar[ur]\ar[ur,shift left=2]\ar[ur,shift right=2]&
      X_{1}'\ar[ur,shift left]\ar[ur,shift right]&
      X_{0}'\ar[ur]&
      {}
    \end{tikzcd}
  \end{equation*}
  where each \(X_{\bullet}'\)
  is a coproduct of contractible condensed sets.
  By \cref{hcov}, this diagram shows that \(X\) is a retract
  of \(\lvert X_{\bullet}'\rvert\).
  Hence the desired result follows
  from \cref{contract}
  and the fact that
  condensed animas
  satisfying the conditions of \cref{contract}
  are closed under colimits.
\end{proof}

\subsection{Discreteness}\label{ss:yes}

We come back to the original question
of when \(K(A)^{\sol}\) is discrete
for a condensed \(\RR\)-algebra~\(A\).
First,
recall that
\(K^{\cn}(A)\) for a connective \(\E_{1}\)-ring~\(A\)
is defined as the group completion
of the core~\((\X)^{\simeq}\)
of the \(\infty\)-category
of finitely generated projective modules \(\Vect(A)\).
We use another model:

\begin{definition}\label{xcvyip}
  For a connective \(\E_{1}\)-ring~\(A\),
  we write \(\Vect'(A)\)
  for the full subanima of
  \(\NN\times{\Vect(A)^{\simeq}}\)
  spanned by
  objects of the form \((r,A^{r})\).
  This has an \(\E_{\infty}\)-monoid structure
  induced from
  those on the factors.
\end{definition}

\begin{lemma}\label{xr12fh}
  For a connective \(\E_{1}\)-ring~\(A\),
  the map
  \(\Vect'(A)\to\Vect(A)^{\simeq}\)
  of \(\E_{\infty}\)-monoids
  is an equivalence after \(\tau_{\geq1}(\X)^{\gp}\).
\end{lemma}

\begin{proof}
To simplify the notation,
  we write \(Y\to X\) for the map.
  We claim that
  \begin{equation*}
    \begin{tikzcd}
      Y^{\gp}\ar[r]\ar[d]&
      X^{\gp}\ar[d]\\
      \pi_{0}(Y^{\gp})\ar[r]&
      \pi_{0}(X^{\gp})
    \end{tikzcd}
  \end{equation*}
  is a pullback square in \(\Cat{Ani}\).
  All objects are \(\ZZ\)-local here
  so it suffices to show that
  \begin{equation*}
    H_{*}(Y^{\gp};\ZZ)
    \to
    H_{*}(\pi_{0}(Y^{\gp})\times_{\pi_{0}(X^{\gp})}X^{\gp};\ZZ)
  \end{equation*}
  is an isomorphism for \({*}\geq0\).
  Note that
  as a graded \(\ZZ\)-algebra,
  \(H_{*}(Y^{\gp};\ZZ)\)
  is obtained from \(H_{*}(Y;\ZZ)\)
  by inverting all \(y\in Y\).
  The same holds for~\(X\),
  but in this case,
  we only need to invert \(x\in X\)
  that comes from~\(Y\)
  since every finitely generated projective module
  is a retract of a finite free module.
  This description proves the desired claim.
\end{proof}

\begin{proposition}\label{xufs5h}
  Let \(A\) be a static condensed ring
  such that \(\SS[\GL_{r}(A)^{\times}]^{\sol}\)
  is discrete for \(r\geq0\).
  Then \((\tau_{\geq1}K(A))^{\sol}\) is discrete.
\end{proposition}

\begin{proof}
By \cref{xr12fh},
  it suffices to prove that
  \((\Vect'(A)^{\gp})^{\sol}\) is discrete.
  By \cref{bd},
  it suffices to show that
  \(\SS[\Vect'(A)^{n}]^{\sol}\)
  is discrete
  for \(n\geq0\).
  Since \(\Vect'(A)\simeq\coprod_{r\geq0}B{\GL_{r}(A)}\),
  it is equivalent to showing that
  \(\SS[B{\GL_{r_{1}}(A)}\times\dotsb\times B{\GL_{r_{n}}(A)}]^{\sol}\)
  is discrete
  for \(r_{1}\), \dots, \(r_{n}\geq0\).
  Therefore,
  it suffices to show that
  \(\SS[B{\GL_{r}(A)}]^{\sol}\) is discrete for \(r\geq0\),
  which follows from the bar resolution
  and our assumption.
\end{proof}

We then deduce our main discreteness results
from this:

\begin{proof}[Proof of \cref{dis_ban}]
  We later see that
  \(K_{0}(A)\) is discrete (before solidification);
  see \cref{kap_dis}.
  Therefore,
  it suffices to show that
  \((\tau_{\geq1}K(A))^{\sol}\) is discrete.
  We prove this by checking the assumption of \cref{xufs5h}.
  Since the condition on~\(A\) is closed under taking \(\Mat_{n}\),
  we need to show that
  \(\SS[A^{\times}]^{\sol}\) is discrete,
  but \(A^{\times}\)
  satisfies the condition of \cref{lslc} in this case.
\end{proof}

\begin{remark}\label{xv4il4}
In \cref{dis_ban}, if we moreover assume that
  \(A\) is locally convex,
  we could also rely on Milnor's theorem~\cite{Milnor59},
  which implies that any open subspace of
  a real locally convex vector space
  is homotopy equivalent 
  to a CW~complex.
\end{remark}

\begin{proof}[Proof of \cref{dis_alg}]
  By \cref{tr}, we can assume that \(A\) is static.

  We first prove that \((\tau_{\geq1}K(A))^{\sol}\) is discrete
  using \cref{xufs5h}.
  Since \(\Mat_{n}(A)\) again satisfies the same condition,
  we wish to show that
  \(\SS[A^{\times}]^{\sol}\) is discrete.
  We write \(A^{\delta}\) for the discrete \(\RR^{\delta}\)-algebra
  before the base change.
  By writing \(A^{\delta}\) as a filtered colimit
  of finitely presented \(\RR^{\delta}\)-algebras,
  we can assume that \(A^{\delta}\) is of finite presentation.
  We have a surjection
  \(\RR^{\delta}\langle T_{1},\ldots,T_{n}\rangle\to A^{\delta}\).
  The algebra
  \(\RR^{\delta}\langle T_{1},\ldots,T_{n}\rangle\)
  has a filtration by total degree
  and we get a filtration on~\(A^{\delta}\)
  from that
  by taking the image.
  Since \(A^{\times}=\injlim_{d\geq0}(\fil^{d}A)\cap A^{\times}\),
  it suffices to show that
  \(\SS[(\fil^{d}A)\cap A^{\times}]^{\sol}\)
  is discrete.
  We here consider the pullback diagram
  \begin{equation*}
    \begin{tikzcd}[column sep=huge]
      (\fil^{d}A)\cap A^{\times}\ar[r,hook,"{a\mapsto(a,a^{-1})}"]\ar[d]&
      \fil^{d}A\times\fil^{d}A\ar[d,"\X\cdot\X"]\\
      {*}\ar[r,"1",hook]&
      \fil^{2d}A\rlap.
    \end{tikzcd}
  \end{equation*}
  Since the right vertical map is polynomial,
  \((\fil^{d}A)\cap A^{\times}\) is a real algebraic variety.
  Hence
  the desired discreteness follows
  from \cref{lslc}.

  Then we proceed to prove that \(K(A)^{\sol}\) is discrete.
  We argue by descending induction on~\(n\)
  to show that
  \((\tau_{\geq n}K(A))^{\sol}\) is discrete
  for any~\(A\) satisfying the condition.
  We have shown that this is the case for \(n=1\).
  We assume the result for~\(n\).
  Since the fundamental theorem of \(K\)-theory says that
  \(K_{n-1}(A)\)
  is a direct summand
  of the cokernel
  of the map \(K_{n}(A[T])\oplus K_{n}(A[T^{-1}])
  \to K_{n}(A[T,T^{-1}])\),
  the result remains true for \(n-1\).
\end{proof}

\section{Comparison with semitopological \texorpdfstring{\(K\)}{K}-theory}\label{s:stop}

For an algebraic variety~\(X\) over~\(\CC\),
we can consider
its \emph{semitopological \(K\)-theory}
\begin{equation}
  \label{e:9wyu5}
  K^{\st}(X)
  =
  \lvert K(X\times_{\Spec\CC}\Spec\Cls{C}(\Delta^{\bullet};\CC))\rvert.
\end{equation}
We refer the reader to~\cite{FriedlanderWalker05}
for a survey of this invariant\footnote{More precisely,
  they consider
  its connective version~\(K^{\cn,\st}\).
}.
Blanc~\cite{Blanc16} introduced its noncommutative variant
(in a slightly different manner; see \cref{xzpvdw}).
Here,
first in \cref{ss:hi},
we study the relationship
between
solidification and
the homotopification process
as in the right-hand side of \cref{e:9wyu5}.
In \cref{ss:stop_cn},
we prove the real version of \cref{com_cat} using that.

\begin{remark}\label{xzpvdw}
  Blanc~\cite{Blanc16} did not define the semitopological \(K\)-theory
  of a \(\CC\)-linear \(\infty\)-category~\(\cat{C}\) in this way,
  but the argument in the proof of~\cite[Theorem~2.3]{AntieauHeller18}
  shows that his invariant coincides with the geometric realization of
  \(K(\cat{C}\otimes_{\CC}\Cls{C}(\Delta^{\bullet};\CC))\).
\end{remark}

\subsection{Recollection: homotopification}\label{ss:hi}

One major caveat when considering~\(K^{\st}\) in this context
is that
\(K(\Cls{C}([0,1];\CC))\) for example is \emph{not} the value
of the condensed spectrum \(K(\CC)\) at \([0,1]\);
hence we need to work with general presheaves,
i.e.,
functors
\(F\colon\Cat{Cpt}^{\op}\to\Cat{Sp}\)
that are not condensed.
Still,
under mild conditions,
we can sheafify this to get
the corresponding condensed spectrum:

\begin{lemma}\label{xpog9z}
  Let \(F\colon\Cat{Cpt}^{\op}\to\cat{D}\)
  be an accessible functor
  to a presentable \(\infty\)-category.
  Then there is a universal condensed object
  \(F'\colon\Cat{Cpt}^{\op}\to\cat{D}\)
  (see \cref{xj8jo5}) under~\(F\).
\end{lemma}

\begin{proof}
We take an infinite strong limit cardinal~\(\kappa\)
  such that \(F\) preserves \(\cf(\kappa)\)-filtered colimits
  and
  finite products commute with \(\cf(\kappa)\)-filtered colimits
  in~\(\cat{D}\).
  We restrict~\(F\) to \(\Cat{EDis}_{\kappa}\)
  and sheafify it there
  to get a \(\kappa\)-condensed object,
  which is the desired one
  by \cref{x9djsw}.
\end{proof}

\begin{definition}\label{xigu8f}
  In the situation of \cref{xpog9z},
  we write \(F^{\con}\) for
  the universal condensed object~\(F'\)
  and call it the \emph{condensation} of~\(F\).
\end{definition}

We then talk about
homotopy invariance and prove the well-known fact:

\begin{definition}\label{xg4xv2}
  An accessible functor \(F\colon\Cat{Cpt}^{\op}\to\cat{D}\)
  to a presentable \(\infty\)-category
  is called \([0,1]\)-homotopy invariant
  if the tautological map
  \(F(X)\to F(X\times[0,1])\) is an equivalence
  for \(X\in\Cat{Cpt}\).
\end{definition}

\begin{example}\label{xvqpc9}
  A solid spectrum regarded as
  an accessible functor
  \(F\colon\Cat{Cpt}^{\op}\to\Cat{Sp}\)
  by \cref{xj8jo5}
  is \([0,1]\)-homotopy invariant:
  This follows from
  the fact that \(\SS[[0,1]]\to\SS\)
  solidifies to an equivalence,
  which follows from \cref{x9jpqc}.
\end{example}

\begin{proposition}\label{xwyhup}
  For an accessible functor \(F\colon\Cat{Cpt}^{\op}\to\cat{D}\)
  to a presentable \(\infty\)-category,
  there is a universal \([0,1]\)-homotopy invariant accessible functor
  under~\(F\),
  which we write \(L_{[0,1]}F\).
  Concretely,
  \(L_{[0,1]}F\)
  is given as the geometric realization
  of \(F(\X\times\Delta^{\bullet})\).
\end{proposition}

\begin{lemma}\label{special}
  Let \(I\) be an object of an \(\infty\)-category~\(\cat{C}\)
  with finite products
  and \(0\) and \(1\colon{*}\to I\)
  are morphisms.
  Suppose that
  there is a map
  \(m\colon I\times I\to I\)
  and homotopies
  \(m(0,\X)\simeq0\)
  and \(m(1,\X)\simeq{\id}\).
  For a functor \(F\colon\cat{C}^{\op}\to\cat{D}\),
  the following conditions are equivalent:
  \begin{conenum}
    \item\label{i:sha}
      The tautological map
      \(F(C)\to F(C\times I)\)
      is an equivalence for any~\(C\in\cat{C}\).
    \item\label{i:rei}
      For any object \(C\in\cat{C}\),
      the maps
      \(F(C\times I)\rightrightarrows F(C)\)
      induced from~\(0\) and~\(1\)
      are homotopic.
  \end{conenum}
\end{lemma}

\begin{proof}
  The implication \(\text{\cref{i:sha}}\Rightarrow\text{\cref{i:rei}}\)
  is clear.
  We assume~\cref{i:rei}
  and wish to prove~\cref{i:sha}.
  We write \(p\colon I\to{*}\)
  for the tautological map.
  For \(C\in\cat{C}\),
  the diagram
  \begin{equation*}
    \begin{tikzcd}[row sep=small,column sep=huge]
      {}&
      C\ar[dr,bend left=15,"{\id}_{C}\times0"]&
      {}\\
      C\times I
      \ar[ur,bend left=15,"{\id}_{C}\times p"]
      \ar[r,shift left,"{\id}_{C\times I}\times0"]
      \ar[r,shift right,"{\id}_{C\times I}\times1"']
      \ar[rr,bend right=25,"{\id}_{C\times I}"]&
      C\times I\times I\ar[r,"{\id}_{C}\times m"]&
      C\times I\rlap,
    \end{tikzcd}
  \end{equation*}
  whose upper and lower \(2\)-cells
  are commutative,
  shows that
  \(F({\id}_{C}\times0)\)
  is
  a right inverse to
  \(F({\id}_{C}\times p)\)
  and therefore
  \(F({\id}_{C}\times p)\) is an equivalence.
\end{proof}

\begin{proof}[Proof of \cref{xwyhup}]
  In this proof,
  we write \(L_{[0,1]}F\)
  for the geometric realization
  of \(F(\X\times\Delta^{\bullet})\)
  and prove that it satisfies the desired universal property.

  When \(F\) is \([0,1]\)-homotopy invariant,
  it is clear that
  the tautological map \(F\to L_{[0,1]}F\) is an equivalence.
  Hence it suffices to show that
  \(L_{[0,1]}F\)
  is \([0,1]\)-homotopy invariant for any~\(F\).

  By \cref{special},
  it suffices to show that
  for any compactum~\(X\)
  the maps
  \(L_{[0,1]}F(X)\rightrightarrows L_{[0,1]}F(X\times[0,1])\)
  induced by~\(0\) and \(1\in[0,1]\)
  are homotopic.
  This comes from the existence
  of a cosimplicial homotopy
  between~\(0\) and
  \(1\colon\Delta^{\bullet}\to\Delta^{\bullet}\times[0,1]\).
\end{proof}

\begin{remark}\label{x57xws}
  Under the situation of \cref{special},
  Morel--Voevodsky constructed in~\cite[Section~2.3]{MorelVoevodsky99}
  a cosimplicial object
  \(I^{\bullet}\) (which does not depend on the map~\(m\))
  such that the geometric realization of
  \(F(\X\times I^{\bullet})\)
  computes the \(I\)-homotopification.
  Beware that
  this one for \(I=[0,1]\) is \emph{not} equivalent
  to the simplicial object we have used in \cref{xwyhup}.
\end{remark}

\subsection{Comparison}\label{ss:stop_cn}

We prove the real version of \cref{com_cat}:

\begin{theorem}\label{hi_alg}
  Let \(A\) be the base change of
  a (discrete) associative ring over~\(\RR^{\delta}\)
  along \(\RR^{\delta}\to\RR\).
  Then \(K(A)^{\sol}\),
  which is discrete by \cref{dis_alg},
  is canonically identified with
  the geometric realization of
  \(K(A\otimes_{\RR}\Cls{C}(\Delta^{\bullet};\RR))\).
\end{theorem}

\begin{remark}\label{xgwdr7}
  By~\cite{Konovalov22}
  and \cref{tr},
  we can generalize \cref{hi_alg}
  to any connective \(\E_{1}\)-ring.
  We can also generalize it
  to any \(\RR\)-linear \(\infty\)-category
  with a bounded weight structure
  by writing it as a filtered colimit
  of \(\Perf\) of connective \(\E_{1}\)-rings.

  We do not know if \cref{hi_alg} remains true for
  arbitrary \(\RR\)-linear \(\infty\)-categories~\(\cat{C}\).
\end{remark}

\begin{proposition}\label{x286fd}
  Let \(X\) be a compactly generated Hausdorff space
  that satisfies the condition of \cref{lslc}.
  Then \(F=\Sigma^{\infty}_{+}\circ j(X)\colon\Cat{Cpt}^{\op}\to\Cat{Sp}\)
  satisfies the following:
  \begin{conenum}
    \item\label{i:cart}
      The map
      \(F(X_{1}\amalg\dotsb\amalg X_{n})\to F(X_{1})\oplus\dotsb\oplus F(X_{n})\)
      is an equivalence
      for any~\(X_{1}\), \dots,~\(X_{n}\).
    \item\label{i:compare}
      The morphism \(L_{[0,1]}F\to(F^{\con})^{\sol}\)
      induced by~\cref{xvqpc9}
      is an equivalence at~\(*\).
  \end{conenum}
\end{proposition}

\begin{proof}
  Since \cref{i:cart} is clear,
  we prove~\cref{i:compare} using \cref{lslc}.
  The morphism in question is
  \begin{equation*}
    L_{[0,1]}F
    \to
    (F^{\con})^{\sol}
    \simeq
    \SS[X]^{\sol}
    \simeq
    \SS[X]^{\dis}
    \simeq
    \SS[X^{\dis}].
  \end{equation*}
Hence we are reduced to an unstable question;
  it suffices to show that
  the geometric realization of
  \begin{equation*}
    \Map(\Delta^{\bullet},X)
    \to
    \Map(\Delta^{\bullet},X^{\dis})
  \end{equation*}
  is an equivalence.
  Note that the target is constant
  by \cref{cw_indis}.
  For this,
  we take \(X_{\bullet}\)
  and \(X_{\bullet}'\) in the proof of \cref{lslc}.
  By the same argument,
  it suffices to show that
  the geometric realization of
  \begin{equation*}
    \Map(\Delta^{\bullet},X'_{n})
    \to
    \Map(\Delta^{\bullet},(X'_{n})^{\dis})
    \simeq
    (X'_{n})^{\dis}
  \end{equation*}
  is an equivalence for \(n\geq0\).
  Since \(X'_{n}\) is a coproduct of contractible condensed sets,
  this follows from \cref{x6lbvd} below.
\end{proof}

\begin{lemma}\label{x6lbvd}
  Let \((X_{i})_{i}\) be a family of contractible condensed animas
  and \(X\) denote the coproduct.
  Then the geometric realization of the map
  \begin{equation*}
    \Map(\Delta^{\bullet},X)
    \to
    \Map(\Delta^{\bullet},X^{\dis}),
  \end{equation*}
  where \(X^{\dis}\) exists by~\cref{i:c_dis}
  of \cref{contract},
  is an equivalence.
\end{lemma}

\begin{proof}
  For each~\(i\),
  we fix a point \(x_{i}\in X_{i}\) and
  a homotopy between \(x_{i}\) and \({\id}_{X_{i}}\).
For a condensed anima~\(T\),
  we consider the composite
  \begin{equation}
    \label{e:phcf6}
    \Map(T,X)
    \to
    \Map(T\times[0,1],X\times[0,1])
    \xrightarrow{h}
    \Map(T\times[0,1],X).
  \end{equation}
  By definition,
  this restricts to the composite
  \(\Map(T,X)\to\Map(T,X^{\dis})\xrightarrow{x}\Map(T,X)\)
  at \(0\in[0,1]\) and
  to \(\id_{\Map(T,X)}\) at \(1\in[0,1]\).

  We write \(F\colon\Cat{Cpt}^{\op}\to\Cat{Ani}\)
  for the accessible functor corresponding to~\(X\).
  Then \cref{e:phcf6} gives a map
  \(F(\X)\to F(\X\times[0,1])\)
  and it induces a map
  \((L_{[0,1]}F)(*)\to(L_{[0,1]}F)([0,1])\)
  that is an equivalence since \(L_{[0,1]}F\) is \([0,1]\)-homotopy invariant.
  Therefore,
  the restrictions to~\(0\) and \(1\in[0,1]\)
  both give its inverse.
  From this and the formula of \(L_{[0,1]}\) in \cref{xwyhup},
  we get the desired result.
\end{proof}

\begin{lemma}\label{x4b50v}
  The class of accessible functors
  \(\Cat{Cpt}^{\op}\to\Cat{Sp}\)
  satisfying~\cref{i:cart,i:compare} of \cref{x286fd}
  is closed under colimits.
\end{lemma}

\begin{proof}
  The class of accessible functors satisfying~\cref{i:cart}
  is closed under colimits.
  Note also that
  \(L_{[0,1]}\) preserves colimits.
  The result follows from the observation that
  when we only consider functors~\(F\)
  satisfying~\cref{i:cart},
  the construction
  \(F\mapsto(F^{\con})^{\sol}(*)\)
  preserves colimits.
\end{proof}

\begin{proof}[Proof of \cref{hi_alg}]
  We write \(K(j(A))\) for the accessible functor
  \(\Cat{Cpt}^{\op}\to\Cat{Sp}\)
  given by
  \(T\mapsto K(\Cls{C}(T;A))\)
  and similarly for its variants.
  We wish to show that
  \(K(j(A))\) satisfies~\cref{i:compare} of \cref{x286fd}.

  We first prove this for \(\tau_{\geq1}K(j(A))\).
  By \cref{x4b50v}
  and the same argument as in the proof of \cref{xufs5h},
  it suffices to show that
  \(\SS[B{\GL_{r_{1}}}(j(A))\times\dotsb\times B{\GL_{r_{n}}(j(A))}]\)
  for \(r_{1}\), \dots, \(r_{n}\geq0\)
  satisfies~\cref{i:compare} of \cref{x286fd}.
  By using \cref{x4b50v} again,
  it suffices to show that
  \(\SS[{\GL_{r_{1}}}(j(A))\times\dotsb\times{\GL_{r_{n}}(j(A))}]
  =\SS[j({\GL_{r_{1}}}(A)\times\dotsb\times{\GL_{r_{n}}(A)})]\)
  satisfies the same.
  The proof of \cref{dis_alg}
  shows that
  the topological space
  \({\GL_{r_{1}}}(A)\times\dotsb\times{\GL_{r_{n}}(A)}\)
  satisfies the condition of \cref{lslc}
  and therefore \cref{x286fd} applies.

  We then prove this for \(K(j(A))\),
  but the argument in the proof of \cref{dis_alg}
  applies here as well.
\end{proof}

We record the following variant in the Banach case:

\begin{theorem}\label{hi_ban}
  Let \(A\) be a real Banach\footnote{In fact,
    this remains true for real quasi-Banach algebras
    since \(K_{0}\) is still \([0,1]\)-homotopy invariant in this case
    as we will explain in~\cite{k-ros-3}.
  } algebra.
  Then \((K^{\cn}(A))^{\sol}\),
  which is discrete by \cref{dis_ban},
  is canonically identified with
  the geometric realization of
  \(K^{\cn}(\Cls{C}(\Delta^{\bullet};A))\).
\end{theorem}

\begin{proof}
  Note that \(K_{0}\) is \([0,1]\)-homotopy invariant
  and hence it suffices to consider \(\tau_{\geq1}K\).
  By the same argument as in the proof of \cref{hi_alg},
  the desired result follows from the fact that
  \({\GL_{r_{1}}}(A)\times\dotsb\times{\GL_{r_{n}}(A)}\)
  satisfies the condition of \cref{lslc}.
\end{proof}

\section{Comparison with operator \texorpdfstring{\(K\)}{K}-theory}\label{s:op}

We prove the real version of \cref{com_ban} in this section.
Usually, the term ``operator \(K\)-theory''
refers to the homotopy groups,
instead of the whole spectrum.
We characterize it using reasonable axioms
and show that \((K^{\cn})^{\sol}\)
satisfies them in \cref{ss:op_cn}.
From that,
we get \(K^{\top}\)
by inverting the Bott element~\(\beta\).
We explain
the geometric construction of~\(\beta\) in \cref{ss:op_bott}.

\subsection{Comparison of the connective part}\label{ss:op_cn}

We compare \(K^{\cn}(A)^{\sol}\),
which is discrete by \cref{dis_ban},
with \(K^{\top}(A)\).

\begin{definition}\label{xw1rsb}
  We write \(\Cat{BAlg}_{\RR}\) for the category
  of real Banach algebras
  and contracting algebra maps.
\end{definition}

\begin{proposition}\label{x9ac0m}
  Let \(\cat{C}\) be the full subcategory
  of \(\Fun(\Cat{BAlg}_{\RR},\Cat{Sp}_{\geq0})_{K^{\cn}/}\)
  spanned by the functors~\(E\) under~\(K^{\cn}\)
  satisfying the following:
  \begin{conenum}
    \item\label{i:zero}
      The morphism
      \(K^{\cn}\to E\)
      induces an equivalence on~\(\pi_{0}\).
    \item\label{i:hi}
      The functor
      \(E\) is homotopy invariant
      in the sense that
      \(E(A)\to E(\Cls{C}([0,1];A))\)
      is an equivalence
      for any Banach algebra~\(A\).
    \item\label{i:exc}
      It is excisive;
      i.e.,
      when
      \begin{equation*}
        \begin{tikzcd}
          A\ar[r]\ar[d]&
          A'\ar[d]\\
          B\ar[r]&
          B'
        \end{tikzcd}
      \end{equation*}
      is a square in \(\Cat{BAlg}_{\RR}\)
      such that the rows are surjective
      with isomorphic kernels,
      it is carried to a cartesian square in \(\Cat{Sp}_{\geq0}\)
      by~\(E\).
  \end{conenum}
  Then \(\cat{C}\) is contractible
  and spanned by \(K^{\cn}(\X)^{\sol}\).
\end{proposition}

Since any reasonable construction
of~\(\tau_{\geq0}K^{\top}(\X)\)
satisfies the conditions in \cref{x9ac0m},
we say that
\(K^{\cn}(\X)^{\sol}\)
is equivalent to \(\tau_{\geq0}K^{\top}(\X)\).

\begin{proof}
  Note that \(K^{\cn}\to(K^{\cn}(\X))^{\sol}\)
  satisfies \cref{i:zero,i:hi,i:exc}
  by \cref{kap_dis,hi_ban,exc},
  respectively.

  We consider \(K^{\cn}\to E\) in~\(\cat{C}\).
  By \cref{i:hi}
  and the universality of~\(K^{\cn}(\X)^{\sol}\) shown in \cref{hi_ban},
  we get a canonical map \(K^{\cn}(\X)^{\sol}\to E\).
  We claim that this is an equivalence.
  We show by induction that
  it induces an equivalence on~\(\pi_{n}\).
  When \(n=0\), this follows from~\cref{i:zero}.
  For any~\(A\),
  from
  \(S^{1}=D^{1}\amalg_{S^{0}}D^{1}\),
  we get the square
  \begin{equation*}
    \begin{tikzcd}
      \Cls{C}(S^{1};A)\ar[r]\ar[d]&
      \Cls{C}(D^{1};A)\ar[d]\\
      \Cls{C}(D^{1};A)\ar[r]&
      \Cls{C}(S^{0};A)\rlap.
    \end{tikzcd}
  \end{equation*}
  By the Dugundji extension theorem~\cite{Dugundji51},
  all the arrows are surjective.
  Hence by \cref{i:hi,i:exc},
  we see that \(\pi_{n}E(\Cls{C}(S^{1};A))\)
  is isomorphic to
  \(\pi_{n}E(A)\oplus\pi_{n+1}E(A)\)
  compatibly with \(K^{\cn}(\X)^{\sol}\to E\).
  Therefore, this map is an equivalence.
\end{proof}

\subsection{Bott inversion}\label{ss:op_bott}

Therefore,
we can construct~\(K^{\top}\)
from \(K^{\cn}(\X)^{\sol}\)
by inverting the Bott element.
We here explain how to concretely construct various Bott elements:

\begin{definition}\label{xu4fsz}
The map \(B\HH^{\times}\to\Omega^{\infty}K^{\cn}(\HH)\)
  determines the map
  \(\Sigma\SS[\HH^{\times}]\to K^{\cn}(\HH)\).
  By solidification,
  we get a map
  \(\Sigma\SS[\HH^{\times}]^{\sol}
  \simeq\Sigma\SS\oplus\Sigma^{4}\SS
  \to K^{\cn}(\HH)^{\sol}\).
  Let \(\beta_{\HH}\in\pi_{4}(K^{\cn}(\HH)^{\sol})\)
  be the image of a generator
  of \(\pi_{4}(\Sigma^{4}\SS)\),
  which is determined up to sign.

  Note that the canonical map
  \(\RR\to\HH\otimes_{\RR}\HH\simeq\Mat_{4}(\RR)\)
  induces an equivalence
  on \(K\)-theory.
  Hence there is a canonical map
  \begin{equation*}
    \pi_{4}(K^{\cn}(\HH)^{\sol})
    \otimes
    \pi_{4}(K^{\cn}(\HH)^{\sol})
    \to
    \pi_{8}(K^{\cn}(\HH\otimes_{\RR}\HH)^{\sol})
    \simeq
    \pi_{8}(K^{\cn}(\RR)^{\sol})
  \end{equation*}
  We write~\(\beta_{\RR}\) for the image
  of \(\beta_{\HH}\otimes\beta_{\HH}\).

  Similarly,
  we write \(\beta_{\CC}\in\pi_{2}(K^{\cn}(\CC)^{\sol})\)
  for the image of a generator
  of \(\pi_{2}(\Sigma^{2}\SS)\)
  under the map
  \(\Sigma\SS[\CC^{\times}]^{\sol}\simeq\Sigma\SS\oplus\Sigma^{2}\SS
  \to K^{\cn}(\CC)^{\sol}\),
  which is determined up to sign.
\end{definition}

\begin{remark}\label{xcm1ia}
  The symbol~\(\beta_{\X}\) just comes from notational convenience
  and we do not claim that there is a uniform way to define them.
\end{remark}

We see that \(\beta_{\RR}\) and~\(\beta_{\CC}\)
are the usual Bott elements:

\begin{proposition}\label{xyrxoa}
  We consider \(\beta_{\RR}\) and~\(\beta_{\CC}\)
  in \cref{xu4fsz}.
  \begin{enumerate}
    \item\label{i:beta_r}
      The element~\(\beta_{\RR}\) is a generator
      of \(\pi_{8}{\ko}\simeq\ZZ\).
    \item\label{i:beta_c}
      The element~\(\beta_{\CC}\) is a generator
      of \(\pi_{2}{\ku}\simeq\ZZ\).
  \end{enumerate}
\end{proposition}

\begin{proof}
First, \cref{i:beta_c}
  follows from the fact
  that the map of animas \(\GL_{1}(\CC)^{\dis}\to{\GL_{\infty}(\CC)^{\dis}}\)
  induces an isomorphism on~\(\pi_{1}\).
  Then \cref{i:beta_r}
  follows from \cref{i:beta_c} and \cref{betas} below.
\end{proof}

\begin{lemma}\label{betas}
  The image of~\(\beta_{\RR}\)
  under the canonical map
  \(\pi_{8}(K^{\cn}(\RR)^{\sol})
  \to
  \pi_{8}(K^{\cn}(\CC)^{\sol})\)
  is \(\pm\beta_{\CC}^{4}\).
\end{lemma}

\begin{proof}
By considering the base change
  \begin{equation*}
    \Mat_{4}(\RR)
    \simeq
    \HH\otimes_{\RR}\HH
    \to
    (\HH\otimes_{\RR}\CC)\otimes_{\CC}(\HH\otimes_{\RR}\CC)
    \simeq
    \Mat_{4}(\CC),
  \end{equation*}
  it is reduced to show that
  the map
  \(\HH\to\HH\otimes_{\RR}\CC\simeq\Mat_2(\CC)\)
  carries \(\beta_{\HH}\) to \(\pm\beta_{\CC}^{2}\)
  on \(\pi_{4}(K^{\cn}(\X)^{\sol})\).
  This follows from the fact
  that the maps of animas \((\HH^{\times})^{\dis}
  \to\GL_{2}(\CC)^{\dis}\to{\GL_{\infty}(\CC)^{\dis}}\)
  induce isomorphisms on~\(\pi_{3}\).
\end{proof}

By using~\cite[Theorem~5.14]{k-ros-1},
which implies that \(K(\RR)\) is connective
as a condensed spectrum,
we see that
we can also invert the Bott element
after solidifying the nonconnective \(K\)-theory:

\begin{theorem}\label{cb}
  For a real Banach algebra~\(A\),
  the condensed spectrum
  \(K(A)^{\sol}[\beta_{\RR}^{-1}]\)
  is discrete and canonically equivalent to \(K^{\top}(A)\).
\end{theorem}

\begin{proof}
  This follows from
  \(K^{\top}(A)\simeq(K^{\cn}(A)^{\sol})[\beta_{\RR}^{-1}]\)
  and
  \cref{x59vdq} below.
\end{proof}

\begin{lemma}\label{x59vdq}
  Let \(M\) be a condensed \(K(\RR)\)-module.
  Since \(K(\RR)\) is connective by~\cite[Theorem~5.14]{k-ros-1},
  there is a map of \(K(\RR)\)-modules
  \(\tau_{\geq n}M\to M\) for \(n\in\ZZ\).
  This map is an equivalence
  after applying \((\X)^{\sol}[\beta_{\RR}^{-1}]\).
\end{lemma}

\begin{proof}
  By \(\injlim_{n}\tau_{\geq n}M\simeq M\),
  it suffices to show that
  when a \(K(\RR)\)-module~\(M\)
  is static,
  \(M^{\sol}[\beta_{\RR}^{-1}]\) vanishes.
  We can see that
  \(\beta_{\RR}\colon\Sigma^{8}M^{\sol}\to M^{\sol}\) is null:
  On~\(M^{\sol}\),
  the action of \(K(\RR)^{\sol}\simeq{\ko}\)
  factors through that of \((\pi_{0}K(\RR))^{\sol}\simeq\ZZ^{\sol}\simeq\ZZ\).
  Since \(\beta_{\RR}\) is killed
  by \({\ko}\to\ZZ\),
  the desired result follows.
\end{proof}

\begin{remark}\label{x57guz}
Instead of picking an element,
  we can get~\(K^{\top}\) via \(\KU\)-localization.
  From an argument similar to the proof of \cref{cb},
  we see that
  \(L_{\KU}(K(A)^{\sol})\)
  is discrete and is equivalent to \(K^{\top}(A)\)
  for a real Banach algebra~\(A\).
\end{remark}

\section{Question: discreteness before solidification}\label{s:ros}

Recall that
a \emph{normed ring}
is a ring with a function \(\lVert\X\rVert\colon A\to[0,\infty)\)
where
\(\lVert x\rVert=0\) if and only if \(x=0\),
\(\lVert -x\rVert=\lVert x\rVert\),
\(\lVert x+y\rVert\leq\lVert x\rVert+\lVert y\rVert\),
\(\lVert 1\rVert\leq1\), and
\(\lVert xy\rVert\leq\lVert x\rVert\lVert y\rVert\).
We say that it is a \emph{Banach ring}
if it is complete.
We consider the following:

\begin{conjecture}\label{yos}
  Let \(A\) be a Banach ring.
  \begin{enumerate}
    \item\label{i:yos_pos}
      For \(n>0\)
      that is invertible in~\(A\),
      the condensed abelian group
      \(K_{*}(A;\ZZ/n)\) is discrete
      for~\(*\geq0\).
    \item\label{i:yos_neg}
      The condensed abelian group
      \(K_{*}(A)\) is discrete
      for \({*}\leq0\)
      when \(A\) is a C*-algebra.
  \end{enumerate}
\end{conjecture}

As a piece of evidence,
in \cref{ss:kaplansky},
we prove, using a standard argument,
the stronger statement that
\(K_{0}(A)\) is discrete
for every Banach ring.
Several remarks are in order:

\begin{remark}\label{xuopy6}
As was explained in~\cite[Session~11]{Masterclass},
  as a consequence of Gabber rigidity,
  we see that \cref{i:yos_pos} of \cref{yos}
  is true in the commutative case.
  Therefore,
  we get the commutative case of~\cref{i:ros_pos} of \cref{ros},
  which was originally due to
  Prasolov~\cite{Prasolov84}
  and Fischer~\cite{Fischer90}.

  In light of this,
  this discreteness conjecture in positive degrees
  can be interpreted as
  asking for some rigidity statement of \(K^{\cn}\)
  for noncommutative algebras.
\end{remark}

\begin{remark}\label{xjx47v}
  By~\cite{k-ros-3},
  \(K_{-1}(A)\) is discrete
  for every Banach algebra~\(A\).
  The first arXiv version of this article
  omitted the C*-algebra hypothesis
  from \cref{i:yos_neg} of \cref{yos}.
  That stronger statement is false;
  see \cref{ss:counterexample}.
\end{remark}

\begin{remark}\label{x9fui6}
On~\cite[page~85]{Rosenberg97},
  Rosenberg explicitly stated
  that \cref{i:ros_neg} of \cref{ros},
  even for \({*}=-1\), should be
  considered specific to C*-algebras rather than real Banach algebras.
  This claim is based on the fact that
  \(\NK_{0}(\X)=\coker(K_{0}(\X)\to K_{0}(\X[T]))\)
  is not \([0,1]\)-homotopy invariant
  in the category of real Banach algebras;
  see~\cite[Example~2.8]{Rosenberg97}.
  However, this issue can be circumvented
  by avoiding the use of
  the fundamental theorem of \(K\)-theory;
  we prove the \([0,1]\)-homotopy invariance
  of~\(K_{-1}\)
  for all real Banach algebras in~\cite{k-ros-3}.
  However,
  we prove in \cref{ss:counterexample}
  that the Banach algebra analog
  of \cref{i:ros_neg} of \cref{ros}
  fails for \({*}\leq-2\).
\end{remark}

What is nice about \cref{yos}
is that it is only about \(K\)-theory
of \(\Cls{C}(\beta I;A)\) for a set~\(I\),
which is ring-theoretically nicer than \(\Cls{C}([0,1];A)\);
e.g., see~\cite[Theorem~5.14]{k-ros-1} for \(A=\RR\).
Since \(\tau_{\geq0}K^{\top}(A)\)
for a real Banach algebra~\(A\)
is a discretization of~\(K^{\cn}(A)\)
by \cref{dis_ban},
we see that
\cref{i:yos_pos} of \cref{yos}
implies
\cref{i:ros_pos} of \cref{ros}.
We also prove that \cref{i:yos_neg} of \cref{yos}
implies \cref{i:ros_neg} of \cref{ros};
this is nontrivial
since \(K(\Cls{C}([0,1];A))\)
is not the value of the condensed spectrum \(K(A)\)
at \([0,1]\).
More precisely,
we prove the following implications
in \cref{ss:d_c}:

\begin{theorem}\label{yros}
  Let \(A\) be a real Banach algebra
  and \(n\geq1\) an integer.
  Assume that
  the condensed abelian groups
  \(K_{-n}(A)\),
  \(K_{-n+1}(A)\),
  and \(K_{-n}(\Cls{C}([0,1];A))\)
  are discrete.
  Then the tautological map
  \(K_{-n}(A)(*)\to K_{-n}(\Cls{C}([0,1];A))(*)\)
  of abelian groups
  is an isomorphism.
\end{theorem}

\subsection{Example: Kaplansky condensed rings}\label{ss:kaplansky}

We here show that many rings
have discrete~\(K_{0}\).
In particular, we see that \cref{i:yos_neg} of \cref{yos} holds
for \({*}=0\).
Formally, we consider the following class:

\begin{definition}\label{xrlgpv}
  We call a (static) condensed ring~\(A\) \emph{Kaplansky}
  if
  for any ultrafilter~\(\mu\) on any set~\(I\),
  the inclusion
  \(A^{\times}\hookrightarrow A\) has a right lifting property with respect to
  \(\{\mu\}\to(\beta I)^{\dag}_{\mu}\);
  i.e., if
  any element of \(\injlim_{\mu(J)=1}A(\beta J)\)
  whose image in \(A(\{\mu\})\)
  is invertible
  lifts to an invertible of \(A(\beta J)\)
  for some \(J\subset I\) with \(\mu(J)=1\).
\end{definition}

\begin{remark}\label{1437d7fff5}
  The terminology in \cref{xrlgpv} comes from
  Kaplansky's famous article~\cite{Kaplansky47},
  where he introduced the notion of a \emph{Q-ring},
  which is basically a topological ring such that \(A^{\times}\subset A\)
  is open\footnote{He considered
    the nonunital version of this notion
    and the letter ``Q'' stands for
    quasiregular elements.
  }.
\end{remark}

\begin{example}\label{xjruob}
  Consider a topological algebra~\(A\) that is Hausdorff.
  When \(A^{\times}\subset A\) is open,
  its associated condensed ring is Kaplansky.
  To see this,
  we observe that
  any open immersion \(U\hookrightarrow K\)
  to a compactum
  has a right lifting property with respect to
  \(\{\mu\}\to(\beta I)^{\dag}_{\mu}\).
  For example,
  any Banach ring is Kaplansky.
\end{example}

\begin{example}\label{b978381acf}
  The polynomial ring~\(\RR[T]\) is not Kaplansky;
  intuitively, this is because there is no neighborhood of~\(1\)
  consisting of invertible elements.
  Formally,
  we consider the map \([0,1]\to\RR[T]\)
  given by \(t\mapsto1+tT\).
  Then for any nonprincipal ultrafilter~\(\mu\) on~\([0,1]\)
  with limit~\(0\),
  we cannot lift this \(\{\mu\}\to\RR[T]^{\times}\)
  to \((\beta[0,1])_{\mu}^{\dag}\).
\end{example}

We prove the following:

\begin{theorem}\label{kap_dis}
  The condensed abelian group
  \(K_0(A)\) is discrete
  for any Kaplansky condensed ring~\(A\).
\end{theorem}

To prove this,
we digress and review
some facts about radical ideals:

\begin{definition}\label{14430daef5}
  A nonunital associative ring~\(I\) is called
  \emph{radical}
  if for any \(a\in I\)
  there is \(b\in I\) such that \(a+b+ab=0\).
  A two-sided ideal of an associative ring is called
  \emph{radical}\footnote{Beware that
    there is a different notion of radicality
    in commutative algebra.
  } if it is radical regarded
  as a nonunital ring.
\end{definition}

\begin{example}\label{x7jbbx}
Let \(A\to B\) be a surjection of rings.
  Then \(\ker(A\to B)\) is radical if and only if
  \begin{equation*}
    \begin{tikzcd}
      A^{\times}\ar[r,hook]\ar[d]&
      A\ar[d]\\
      B^{\times}\ar[r,hook]&
      B
    \end{tikzcd}
  \end{equation*}
  is cartesian.
\end{example}

The notion of Kaplansky ring is
related to that of radicality as follows:

\begin{proposition}\label{4546bdb5d6}
  Let \(I\subset A\) be a radical ideal.
  Then
  finitely generated projective \(A\)-modules~\(P\)
  and~\(P'\)
  are isomorphic
  if and only if their base changes to \(A/I\) are isomorphic.
\end{proposition}

\begin{proof}
  By the projectivity of~\(P\),
  we can lift
  \(P\to A/I\otimes_{A}P\simeq A/I\otimes_{A}P'\)
  to \(P\to P'\).
  Then by Nakayama,
  this is surjective.
  Then by the projectivity of~\(P'\),
  it has a section.
  Again by Nakayama,
  it is a surjection.
  Hence it is an isomorphism.
\end{proof}

\begin{corollary}\label{xijy8l}
  When \(I\) is radical,
  the map \(K_0(A)\to K_0(A/I)\) is injective.
\end{corollary}

\begin{proposition}\label{xogmh9}
  A condensed ring~\(A\) is Kaplansky
  if and only if \(\ker(A((\beta I)_{\mu}^{\dag})\to A(\{\mu\}))\)
  is radical
  for any ultrafilter~\(\mu\) on any set~\(I\).
\end{proposition}

\begin{proof}
  This follows from \cref{x7jbbx}.
\end{proof}

\begin{proof}[Proof of \cref{kap_dis}]
  By \cref{dis},
  we have to show that
  \begin{equation*}
    K_{0}(A((\beta I)_{\mu}^{\dag}))
    \to
    K_{0}(A(\{\mu\}))
  \end{equation*}
  is bijective for any ultrafilter~\(\mu\) on any set~\(I\).
  Since it has a splitting,
  it suffices to show the injectivity,
  which follows from \cref{xijy8l,xogmh9}.
\end{proof}

For future reference,
we note the following analog of~\cite[Theorem~1]{Kaplansky47}:

\begin{corollary}\label{5c9f6925c3}
  If a condensed ring~\(A\) is Kaplansky,
  so is \(\Mat_{r}(A)\) for \(r\geq0\).
\end{corollary}

\begin{proof}
  By \cref{xogmh9},
  this follows from the fact
  that
  \(\Mat_{r}(I)\) is radical
  for a radical nonunital ring~\(I\),
  which is reduced to the case when \(r=2\).
\end{proof}

\subsection{From discreteness to homotopy invariance}\label{ss:d_c}

We recall the following from~\cite[Section~3]{k-ros-1}:

\begin{definition}\label{xqdxy6}
  A \emph{blowup square} of compacta is
a bicartesian square
  \begin{equation}
    \label{e:d0d49}
    \begin{tikzcd}
      Z'\ar[r,hook]\ar[d]&
      X'\ar[d]\\
      Z\ar[r,hook]&
      X
    \end{tikzcd}
  \end{equation}
  whose horizontal maps are closed embeddings.
\end{definition}

\begin{theorem}\label{x87vf0}
  Let \(E\colon\Cat{Cpt}^{\op}\to\Cat{Sp}\) be a functor.
  We write \(F\) for the functor
  given by \(F(\X)=E(\X\times[0,1])\).
  We assume the following:
  \begin{conenum}
    \item\label{i:aleph1}
      The functor~\(E\)
      preserves \(\aleph_{1}\)-filtered colimits.
    \item\label{i:add}
      For compacta~\(X_{1}\), \dots,~\(X_{n}\),
      the map
      \(E(X_{1}\amalg\dotsb\amalg X_{n})
      \to E(X_{1})\oplus\dotsb\oplus E(X_{n})\)
      is an equivalence.
    \item\label{i:n}
Any blowup square of compacta is carried by~\(E\)
      to a square
      whose total fiber has trivial~\(\pi_{0}\) and~\(\pi_{1}\).
    \item\label{i:dis_e}
      The condensation~\(\pi_{0}E^{\con}\)
      and \(\pi_{1}E^{\con}\) are discrete.
    \item\label{i:dis_f}
      The condensation~\(\pi_{0}F^{\con}\) is discrete.
  \end{conenum}
  Then the tautological map \(\pi_{0}E(*)\to\pi_{0}E([0,1])=\pi_{0}F(*)\)
  is an isomorphism.
\end{theorem}

This implies \cref{yros}:

\begin{proof}[Proof of \cref{yros}]
  It suffices to check that
  the conditions of \cref{x87vf0} are satisfied
  for \(F=\Sigma^{n}K(\Cls{C}(\X;A))\).
  Our assumptions directly imply
  \cref{i:dis_e,i:dis_f}.
  The metrizability of~\(A\)
  implies~\cref{i:aleph1}
  and we can also directly check~\cref{i:add}.
  Therefore, it remains only to prove~\cref{i:n}.
  For a blowup square~\cref{e:d0d49},
  we obtain the square
  \begin{equation*}
    \begin{tikzcd}
      \Cls{C}(Z';A)&
      \Cls{C}(X';A)\ar[l]\\
      \Cls{C}(Z;A)\ar[u]&
      \Cls{C}(X;A)\rlap.\ar[u]\ar[l]
    \end{tikzcd}
  \end{equation*}
  By the Dugundji extension theorem~\cite{Dugundji51},
  the horizontal maps are surjective.
  Also, their kernels are isomorphic.
  Hence \(K\) maps this to a square
  whose total fiber is \(1\)-connective.
\end{proof}

The rest of this section
is devoted to the proof of \cref{x87vf0}.
The following notion is useful
in the proof:

\begin{definition}\label{x5fh59}
  Let \(\cat{C}\) be a presentable \(\infty\)-category
  and \(F\colon\Cat{Cpt}^{\op}\to\cat{C}\)
  an accessible functor.
  For a locally compact space \(U\),
  we say that \(F\) is \emph{\(U\)-rigid} if
  the map
  \begin{equation*}
    \injlim_{V\subset U} F(V^{+})
    \to
    F(\{\infty\})
  \end{equation*}
  is an equivalence,
  where \(V\) runs over open subsets of~\(U\)
  with compact complement.
  We similarly define the notion of \(U\)-rigidity
  for an accessible functor \(F\colon\Cat{TDis}^{\op}\to\cat{C}\)
  for a locally profinite set~\(U\).
\end{definition}

\begin{example}\label{x6v66d}
  In the situation of \cref{x5fh59},
  the functor~\(F\)
  is \(U\)-rigid
  for any compactum~\(U\).
  In general,
  for a compact subset
  \(K\subset U\),
  it is \(U\)-rigid if and only if
  it is \(U\setminus K\)-rigid.
\end{example}

\begin{example}\label{rig_ret}
  In the situation of \cref{x5fh59},
  consider a local compactum~\(V\),
  which is a retract of~\(U\)
  in the category of local compacta
  and proper maps.
  Then if \(F\) is \(U\)-rigid,
  it is also \(V\)-rigid.
\end{example}

\begin{example}\label{xzdwhz}
  Consider the situation of \cref{x5fh59}.
  \begin{itemize}
    \item
      The functor~\(F\) is \(\NN\)-rigid
      if and only if
      the map
      \(\injlim_{n\geq0}
      F(\NN_{\geq n}\cup\{\infty\})
      \to
      F(\{\infty\})\)
      is an equivalence.
    \item
      The functor~\(F\) is \([0,\infty)\)-rigid
      if and only if the map
      \(\injlim_{x\geq0}
      F([x,\infty])
      \to
      F(\{\infty\})\)
      is an equivalence.
    \item
For an ultrafilter~\(\mu\) on a set~\(I\),
      the \((\beta I\setminus\{\mu\})\)-rigidity condition
      has appeared in \cref{dis}.
  \end{itemize}
\end{example}

\begin{example}\label{xwx7gw}
  Let \(\cat{C}\) be a compactly assembled \(\infty\)-category
  and \(F\) a condensed object,
  regarded as an accessible functor \(F\colon\Cat{TDis}^{\op}\to\cat{C}\)
  by \cref{xj8jo5}.
  If \(F\) is discrete,
  it is \(U\)-rigid for any locally profinite set~\(U\)
  by \cref{xd9foy}.
\end{example}

The following is a key observation:

\begin{proposition}\label{xz91s9}
  Let \(F\colon\Cat{TDis}^{\op}\to\Cat{Sp}\)
  be a functor satisfying the following condition:
  \begin{conenum}
    \item\label{i:y_aleph1}
      It preserves \(\aleph_{1}\)-filtered colimits.
    \item\label{i:y_add}
      It preserves finite coproducts.
    \item\label{i:y_n}
      It carries
      any blowup square of profinite sets
to a square
      whose total fiber has trivial~\(\pi_{0}\).
    \item\label{i:y_dis}
      It is \(\beta I\setminus\{\mu\}\)-rigid
      for any ultrafilter \(\mu\) on a set~\(I\);
      i.e., its condensation is discrete
      (by \cref{dis}).
\end{conenum}
  Then \(\pi_{0}F\) is \(\NN\)-rigid.
\end{proposition}

Our proof here is inspired by
the proof of~\cite[Proposition~3.15]{Complex}.
We start with the following observation:

\begin{lemma}\label{xxuvm5}
  Let \(F\colon\Cat{TDis}^{\op}\to\Cat{Ab}\)
  be a functor preserving \(\aleph_{1}\)-filtered colimits.
  Then
  the map \(F(S)\to F(T)\) is
  surjective for any inclusion
  of profinite sets \(\emptyset\neq T\hookrightarrow S\).
\end{lemma}

\begin{proof}
  By the assumption on~\(F\),
  it suffices to treat the case
  when both \(S\) and~\(T\)
  are of countable weight.
  In that case,
  we have a retraction \(S\to T\)
  according to Halmos~\cite[Theorem~2]{Halmos61}.
\end{proof}

\begin{proof}[Proof of \cref{xz91s9}]
  We first show that
  \(F(\beta\NN\setminus\NN)\to F^{\con}(\beta\NN\setminus\NN)\)
  is an equivalence.
  We consider the square
  \begin{equation*}
    \begin{tikzcd}
      \injlim_{n\geq0}F(\beta\NN\setminus\{0,\ldots,n-1\})\ar[r]\ar[d]&
      F(\beta\NN\setminus\NN)\ar[d]\\
      \injlim_{n\geq0}F^{\con}(\beta\NN\setminus\{0,\ldots,n-1\})\ar[r]&
      F^{\con}(\beta\NN\setminus\NN)\rlap.
    \end{tikzcd}
  \end{equation*}
  The left arrow is an equivalence
  since
  \(\beta\NN\setminus\{0,\ldots,n-1\}\) is
  extremally disconnected for \(n\geq0\).
  The bottom arrow is an equivalence by~\cref{i:y_dis}
  and \cref{xwx7gw}.
  Since \cref{i:y_aleph1} and \cref{xxuvm5} show
  that the top arrow is surjective on \(\pi_{*}\),
  the right arrow is an equivalence.

Then from the blowup square
  \begin{equation*}
    \begin{tikzcd}
      \beta\NN\setminus\NN\ar[r,hook]\ar[d]&
      \beta\NN\ar[d]\\
      \{\infty\}\ar[r,hook]&
      \NN\cup\{\infty\}
    \end{tikzcd}
  \end{equation*}
  and~\cref{i:y_n},
  we see that
  \(\pi_{0}F(\NN\cup\{\infty\})\to\pi_{0}F^{\con}(\NN\cup\{\infty\})\)
  is injective
  and therefore
  the \(\NN\)-rigidity of~\(F^{\con}\),
  which follows from~\cref{i:y_dis} and \cref{xwx7gw},
  implies that of~\(F\).
\end{proof}

We need two more lemmas:

\begin{lemma}\label{xzhgdy}
  For an accessible functor
  \(E\colon\Cat{Cpt}^{\op}\to\Cat{Sp}\),
  consider another accessible functor~\(F\)
  given by \(F(\X)=E(\X\times[0,1])\).
  Suppose that \(\pi_{0}E\), \(\pi_{1}E\), and~\(\pi_{0}F\) are \(\NN\)-rigid
  and \(E\) carries any blowup square of compacta
  to a square whose total fiber has trivial~\(\pi_{0}\).
  Then \(\pi_{0}E\) is \(([0,1]\setminus\{1/2\})\)-rigid.
\end{lemma}

\begin{proof}
We apply the condition to the blowup squares
  \begin{equation*}
    \begin{tikzcd}
      {[0,1]}\times\{\infty\}\ar[r]\ar[d]&
      {[0,1]}\times\NN_{\geq n}^{+}\ar[d]\\
      \{\infty\}\ar[r]&
      ({[0,1]}\times\NN_{\geq n})^{+}
    \end{tikzcd}
  \end{equation*}
  for \(n\geq0\).
  Then the \(\NN\)-rigidity of~\(\pi_{0}E\) and~\(\pi_{0}F\)
  implies
  the \(([0,1]\times\NN)\)-rigidity of~\(\pi_{0}E\).

We apply the condition to the (pointed) blowup square
  \begin{equation*}
    \begin{tikzcd}[column sep=huge]
      \NN_{\geq n}^{+}
      \ar[r,"k\mapsto{(0,k)}"]
      \ar[d,"k\mapsto{(1,k)}"']&
      ({[0,1]}\times\NN_{\geq n})^{+}
      \ar[d,"{(x,k)}\mapsto x+2k+1"]\\
      ({[0,1]}\times\NN_{\geq n})^{+}
      \ar[r,"{(x,k)}\mapsto x+2k"]&
      {[2n,\infty)^{+}}\rlap.
    \end{tikzcd}
  \end{equation*}
  Then
  the \(([0,1]\times\NN)\)-rigidity of~\(\pi_{0}E\)
  and the \(\NN\)-rigidity of~\(\pi_{1}E\)
  implies the \([0,\infty)\)-rigidity of~\(\pi_{0}E\).
  By replacing each term \((\X)^{+}\)
  in this square
  with \((\X\amalg\X)^{+}\),
  we obtain the \(([0,\infty)\amalg[0,\infty))\)-rigidity of~\(\pi_{0}E\),
  which is equivalent to the \(([0,1]\setminus\{1/2\})\)-rigidity.
\end{proof}

\begin{lemma}\label{x1fpfy}
  Let \(F\colon\Cat{Cpt}^{\op}\to\Cat{Sp}\) be an accessible functor
  that carries any blowup square of compacta
  to a square whose total fiber has trivial~\(\pi_{0}\).
  When \(\pi_{0}F\) is \(([0,1]\setminus\{1/2\})\)-rigid,
  the tautological map \(\pi_{0}F(*)\to\pi_{0}F([0,1])\)
  is an equivalence.
\end{lemma}

\begin{proof}
  For \(0\leq a\leq b\leq1\),
  we write \(G(a,b)\)
  for the split cokernel of the map
  \(\pi_{0}F(*)\to\pi_{0}F([a,b])\).
  We wish to prove that
  \(G(0,1)\) is trivial,
  or equivalently,
  any element of \(G(0,1)\) is zero.
  
By applying the condition to the blowup square
  \begin{equation*}
    \begin{tikzcd}
      \{(a+b)/2\}\ar[r]\ar[d]&
      {[a,(a+b)/2]}\ar[d]\\
      {[(a+b)/2,b]}\ar[r]&
      {[a,b]}\rlap,
    \end{tikzcd}
  \end{equation*}
  we see that
  \(G(a,b)\to G(a,(a+b)/2)\oplus G((a+b)/2,b)\)
  is injective.
  This means that
  if there is a nonzero element in \(G(a,b)\),
  it remains nonzero in
  either \(G(a,(a+b)/2)\) or~\(G((a+b)/2,b)\).
  By repeating this process,
  we see that
  there is \(p\in[0,1]\)
  such that the element does not vanish
  in \(\injlim_{p\in[a,b]}G(a,b)\)
  where \([a,b]\) runs
  over the neighborhoods of~\(p\).
  But by
  the \(([0,1]\setminus\{1/2\})\)-
  or \([0,\infty)\)-rigidity
  (which follows from \cref{rig_ret})
  of~\(\pi_{0}E\),
  this group vanishes.
\end{proof}

\begin{proof}[Proof of \cref{yros}]
By applying \cref{xz91s9}
  to~\(E\), \(\Sigma^{-1}E\), and~\(F\),
  we see that
  \(\pi_{0}E\), \(\pi_{1}E\), and~\(\pi_{0}F\) are \(\NN\)-rigid.
  Hence by \cref{xzhgdy},
  \(\pi_{0}E\) is \(([0,1]\setminus\{1/2\})\)-rigid.
  Therefore, the desired result
  follows from \cref{x1fpfy}.
\end{proof}

\subsection{Addendum: a counterexample}\label{ss:counterexample}

We prove that the C*-algebra hypothesis
in \cref{i:yos_neg} of \cref{yos}
cannot be omitted:
The corresponding statement for Banach rings
is false already in degree~\(-2\).
By \cref{yros}
and the discreteness result
for~\(K_{-1}\) in~\cite[Corollary~1.5]{k-ros-3},
it suffices to show the following:

\begin{proposition}\label{cone_disk}
  For \(n\geq3\),
  there is a commutative complex Banach algebra~\(A\)
  for which the tautological injection
  of abelian groups
  \(
    K_{-n+1}(A)
    \to
    K_{-n+1}(\Cls{C}([0,1];A))
  \)
  is not surjective.
\end{proposition}

Our counterexample
is based on the idea
used in~\cite[Sections~6–7]{k-ros-3}.
The point is that
we can choose
the isolated singularity there to be a homogeneous cone.
The geometric input here is the following concrete
instance of Cortiñas–Haesemeyer–Walker–Weibel’s result~\cite{CHWW13}:

\begin{lemma}\label{xi918u}
  Consider \(n\geq2\).
  Let \(k\) be a field of characteristic~\(0\)
  and \(X=\Proj(R)\subset\PP^{n}_{k}\) a smooth hypersurface
  of degree~\(d\).
  Then \(K_{-n+1}(R)\)
  is a \(k\)-vector space of dimension \(\binom{d-1}{n+1}\).
\end{lemma}

\begin{proof}
  By~\cite[Theorem]{CHWW13},
  we have
  \(
    K_{-n+1}(R)
    \simeq
    \bigoplus_{t\geq1}H^{n-1}(X,\Cls{O}_{X}(t)).
  \)
  Since
  \(\Omega^{n-1}_{X}\simeq\Cls{O}_{X}(d-n-1)\),
  we have
  \(
    K_{-n+1}(R)
    \simeq
    \bigoplus_{j=0}^{d-n-2}H^{0}(X,\Cls{O}_{X}(j))^{\vee}
  \)
  by Serre duality.
  Since
  the restriction map
  \(H^{0}(\PP^{n}_{k},\Cls{O}(j))
  \to H^{0}(X,\Cls{O}_{X}(j))\)
  is an isomorphism for \(0\leq j<d\),
  this is a finite \(k\)-vector space of dimension
  \(
    \sum_{j=0}^{d-n-2}\binom{n+j}{n}
    =\binom{d-1}{n+1}
  \).
\end{proof}

\begin{proof}[Proof of \cref{cone_disk}]
  In the situation of \cref{xi918u},
  take \(k=\CC\), \(d=n+2\),
  and choose a smooth hypersurface~\(X\).
  We consider
  \(Z=(\Spec R)^{\ana}\cap\DD^{n+1}\),
  where \(\DD^{n+1}\subset\AA^{n+1}\) is the closed unit polydisk.
  Then \(Z\) is a Stein compact set.
  As in~\cite[Section~6]{k-ros-3},
  we consider \(A=\Cls{A}(Z)\)
  to be the completion
  of the ring of functions
  holomorphic on a neighborhood of~\(Z\)
  with respect to the supremum norm.
  Then \(A\) is a commutative complex Banach algebra.

  By \cref{xi918u},
  we have \(K_{-n+1}(R)\neq0\).
  Since the ring~\(R\) is an excellent domain
  and only singular at~\(\idl{m}\),
  the cone point of~\(\Spec(R)\),
  we can apply~\cite[Proposition~1.6]{Weibel80} to see that
  \begin{equation*}
    K_{*}(R)\to K_{*}(R^{\wedge}_{\idl{m}})
  \end{equation*}
  is an isomorphism
  for \({*}\leq-2\).
  As in~\cite[Section~7]{k-ros-3},
  taking Taylor expansions at~\(\idl{m}\)
  gives a factorization
  \(R\to A\to R^{\wedge}_{\idl{m}}\).
  Therefore,
  the map \(K_{*}(R)\to K_{*}(A)\) is injective
  on that range,
  so \(K_{-n+1}(A)\) is nonzero.

  We then consider
  \(s\colon A\to\Cls{C}([0,1];A)\)
  given by \(a\mapsto t\mapsto a(tz)\).
  If the homotopy invariance of~\(K_{-n+1}\) held,
  the maps~\({\ev_{0}}\),
  \({\ev_{1}}\colon\Cls{C}([0,1];A)\rightrightarrows A\)
  would induce the same map on~\(K_{-n+1}\).
  However,
  \({\ev_{1}}\circ s={\id}\),
  whereas \({\ev_{0}}\circ s\) factors through~\(\CC\)
  and therefore induces zero on~\(K_{-n+1}\).
  This contradicts the nonvanishing of~\(K_{-n+1}(A)\).
\end{proof}

\let\AA\oldAA \let\L\oldL \let\SS\oldSS \let\top\oldtop \bibliographystyle{plain}
 \newcommand{\yyyy}[1]{}

\end{document}